\documentclass[11pt]{article}
\usepackage{mathtools, amsthm, accents, tikz, amssymb, latexsym, amsmath, bm, amscd, amsfonts, array, lmodern, enumerate, stmaryrd, rotating, caption, graphicx, hyperref, float,tikz-cd, color, yfonts}
\usepackage[new]{old-arrows}
\usepackage[outline]{contour}
\usetikzlibrary{calc}
\usepackage[all]{xy}
\CompileMatrices
\hypersetup{nesting=true,debug=true,naturalnames=true}

\def\C{\protect\operatorname{Conf}}
\def\UC{\protect\operatorname{UConf}}

\newtheorem{proposition}{Proposition}[section]
\newtheorem{corollary}[proposition]{Corollary}
\newtheorem{definition}[proposition]{Definition}
\newtheorem{theorem}[proposition]{Theorem}
\newtheorem{remark}[proposition]{Remark}
\newtheorem{example}[proposition]{Example}
\newtheorem{lemma}[proposition]{Lemma}

\begin{document}

\title{Borsuk--Ulam property for graphs II: The $\mathbb{Z}_n$-action}

\author{Daciberg Lima Gon\c{c}alves and Jes\'us Gonz\'alez}

\date{\today}

\maketitle

\begin{abstract}
For a finite group $H$ and connected topological spaces $X$ and $Y$ such that $X$  is  endowed with a free left $H$-action  $\tau$, we provide a geometric condition in terms of the existence of a commutative diagram of spaces (arising from the triple $(X,Y;\tau)$) to decide whether the Borsuk--Ulam property holds for based homotopy classes $\alpha\in[X,Y]_0$, as well as for free homotopy classes $\alpha\in[X,Y]$. Here, a homotopy class $\alpha$ is said to satisfy the Borsuk--Ulam property if, for each of its representatives $f\in\alpha$, there exists an $H$-orbit where $f$ fails to be injective. Our geometric characterization is attained  by constructing an $H$-equivariant map from $X$ to the classical configuration space $F_{|H|}(Y)$. We derive an algebraic condition from the geometric characterisation, and show that former one is in fact equivalent to the latter one when $X$ and $Y$ are aspherical. We then specialize   to the 1-dimensional case, i.e., when $X$ is an arbitrary connected graph, $H$ is cyclic, and $Y$ is either an interval, a circle, or their wedge sum. The graph-braid-group ingredient in our characterizations is then effectively controlled through the use of discrete Morse theory. 
\end{abstract}



\section{Introduction and main result}\label{intro}
In its classical formulation, the Borsuk--Ulam Theorem asserts that, for any continuous map
\begin{equation}\label{classicaln}
f\colon S^n\to \mathbb{R}^n,
\end{equation}
there is a point $x\in S^n$ so that both $x$ and its antipodal $-x$ have the same image under $f$. Such a phenomenon has been intensively studied in the last 15 years within generalized contexts, namely, for maps $f\colon M\to N$ between spaces $M$ and $N$, where $M$ admits a free involution. For instance, the case where $M$ ranges over surfaces or suitable families of 3-manifolds is now reasonably well understood \cite{MR3614297,BGH1,MRBGH,BauHaGoZv,MR3619753,MR4431413,MR2209795,MR2840097}. The case where $N$ has non-trivial homotopy information leads to a more refined problem, as the Borsuk--Ulam question can then have different answers for different homotopy classes\footnote{Unless otherwise noted, spaces are assumed to come equipped with base points which must be preserved by maps between spaces. Likewise, homotopy classes are meant in the based sense.} in $[M,N]$, see \cite{MR3947929, MR4235703, MR01}.

\begin{definition}[\cite{MR3947929}]
Assume $M$ admits a free involution $\tau$. We say that the Borsuk--Ulam property holds for a homotopy class $\alpha\in[M,N]$ if for every representative $f\in\alpha$ there is a point $x\in M$ such that $f(x)=f(\tau\cdot x)$. If the above condition holds for all homotopy classes in $[M,N]$, we say that the triple $(M,\tau,N)$ satisfies the Borsuk--Ulam property.
\end{definition}

Very recently a  complete answer to the Borsuk--Ulam problem  for homotopy  classes of maps 
between two  finite connected graphs $\Gamma$ and~$G$ 
with respect to any   free involution $\tau$ on $\Gamma$ was given in \cite{GG}.


\begin{remark}\label{notaninterval}{\em
In the classical situation~(\ref{classicaln}) with $n=1$, the circle plays no essential role. Indeed, by considering the differences $f(x)-f(\tau\cdot x)$, it can be seen that any map $f\colon \Gamma\to \mathbb{R}$ satisfies the Borsuk--Ulam property. This is also the case if we replace $\mathbb{Z}_2$ by an arbitrary finite group $H$ acting freely on $\Gamma$ (cf. Section \ref{ciclosyarboles}). However, such a pleasant situation changes drastically when $\mathbb{R}$ (or an interval, for that matter) is replaced by a more general graph~$G$ which throughout this introductory section will be assumed not to be homeomorphic to an interval. In particular the configuration spaces 
 $\C_n(G)$ and $\UC_n(G)$ consisting respectively of $n$-tuples $(x_1, \cdots, x_n)$ and of subsets $\{x_1,\cdots, x_n\}$ with $x_i\neq x_j$ for  $i\ne j$, where $n=|H|\geq 2$ (the cardinality of $H$), are both connected.
}\end{remark}

The Borsuk--Ulam property for $(\Gamma,\tau,G)$ as above (with $\tau$ an involution) is described  in \cite{GG}  by the following results. 
\begin{theorem}\label{maintheorem}
If $G$ is not homeomorphic to a circle or to an interval, then the Borsuk--Ulam property fails for all homotopy classes in $[\Gamma,G]$, i.e., for every $\alpha\in[\Gamma,G]$ there is a representative $f\in\alpha$ satisfying $f(x)\neq f(\tau\cdot x)$ for all $x\in\Gamma$.
\end{theorem}

 When $G$ is a circle, the behavior of the Borsuk--Ulam property sits in between Remark~\ref{notaninterval} and Theorem~\ref{maintheorem}. 
 
\begin{theorem}\label{maintheoremII} If $G$ is homeomorphic to a circle $S^1$, then the Borsuk--Ulam property holds for most
of the homotopy classes in $[\Gamma, S^1]$. Explicitly, the Euler characteristic of $\Gamma$ must be even and non-positive, say $\chi(\Gamma)=-2m$, $m\geq0$, so that $[\Gamma, S^1]$ can be identified with $\mathbb{Z}^{2m+1}$. Then, for a suitable such identification,
the homotopy classes of maps $\Gamma \to S^1$ for which the Borsuk--Ulam property fails are precisely the
$(2m + 1)$-tuples $(p, p_1, p_1, p_2, p_2,\ldots, p_m, p_m)$ with $p$ odd (and $p_1,\ldots, p_m$ arbitrary).
\end{theorem}
  
Observe that Theorems~\ref{maintheorem} and~\ref{maintheoremII} can be stated replacing based homotopy classes by free homotopy classes. This is clear in the case of Theorem \ref{maintheorem}, while the case of Theorem~\ref{maintheoremII} follows from the fact that $\pi_1(S^1)$ is abelian, so that based homotopy classes and free homotopy classes coincide.

\smallskip
 One of the main ingredients used in the proof of Theorems \ref{maintheorem}  and \ref{maintheoremII} is  the algebro-topological criterion given by~\cite[Lemma 5]{MR3947929} 
 and \cite[Theorem 2.4]{Acta}, which we now extend, from involutions   and  $\mathbb{Z}_{n}$-actions to general actions by finite groups. Indeed, a property of the Borsuk--Ulam type for  homotopy classes has been considered by Gon\c{c}alves, Guashi and Laass   in \cite{Acta} within the 
 context of maps between two spaces $X$ and $Y$, where a finite group $H$ (not necessarily cyclic)  acts freely on   $X$. Namely, when 
  $\tau_{H}: H\times X \to X$ is a free left action of $H$ on $X$, the definition of the Borsuk--Ulam property for a class $\alpha\in [X,Y]$
with respect to the action $\tau_H$ given in \cite{Acta} goes as follows:
\begin{definition} We say that a free homotopy class $\alpha  \in [X,Y]$ has the Borsuk--Ulam property with respect to the action $\tau_H$ if, for every representative map $f: X \to Y$ of $\alpha$, there exist $h_1, h_2 \in  H$ and $x \in  X$ such that $f(\tau_H(h_1,x)) = f(\tau_H(h_2,x))$ with $h_1\ne  h_2$. Furthermore, we say that the triple 
$(X,\tau_H,Y)$ satisfies the Borsuk--Ulam property if every free  homotopy class of maps from $X$ to $Y$ has the Borsuk--Ulam property with respect to $\tau_H$.
\end{definition}

The main results of this work address the Borsuk-Ulam property for connected graphs in the cases where the tarjet graph is a tree (Theorem \ref{thirdmainintro}), a cycle (Theorem \ref{secondmainH}), and the connected graph $S^1\vee I$ with two vertices and two edges one of which is a loop (Theorem \ref{maintheoremH}). In all cases the free action on the source graph is denoted by $\tau_{H}\colon H\times\Gamma\to\Gamma$, and has $H=\mathbb{Z}_n$ with $n\geq2$.

\begin{theorem}\label{thirdmainintro}
The Borsuk--Ulam property holds for all triples $(\Gamma,\tau_{\mathbb{Z}_n},I)$, where $I=[0,1]$ is the unit interval. On the other hand, when $G$ is a tree not   homeomorphic to $I$, the Borsuk-Ulam property fails for the unique homotopy class in $[\Gamma,G]$.
\end{theorem}

\begin{theorem}\label{secondmainH}
The Borsuk--Ulam property holds for most of the homotopy classes in $[\Gamma,S^1]$. Explicitly,  for a suitable identification of $\hspace{.4mm}[\Gamma,S^1]$ with $\mathbb{Z}^{nm+1}$ (similar to the one in Theorem \ref{maintheoremII}), the homotopy classes of maps $\Gamma\to S^1$ for which the Borsuk--Ulam property fails are precisely the 
$(nm+1)$-tuples $(p,p_1,p_1,\ldots,p_1,p_2,p_2,\ldots,p_2,\ldots,p_m,p_m,\ldots,p_m)$ with $p$ congruent to $1$ mod $n$ (and  $p_1,\ldots,p_m$ arbitrary).
\end{theorem}


\begin{theorem}\label{maintheoremH} 
Assume that the Euler characteristic of $\hspace{.5mm}\Gamma$ is zero. Then
the Borsuk--Ulam property fails for all homotopy classes in $[\Gamma, S^1\vee I]$, i.e., for every $\alpha\in[\Gamma, S^1\vee I]$ 
there is a representative $f\in\alpha$ satisfying
$ f(\tau_{\mathbb{Z}_n}(h_1, x))  \ne    f(\tau_{\mathbb{Z}_n}(h_2, x))$ for  all  $h_1\ne h_2 \in \mathbb{Z}_n$ and  all $x\in\Gamma$.            
\end{theorem}

As in \cite{MR3993193,MR2639841,MR3947929,MR4235703,MR01}, we study the Borsuk--Ulam property for graphs through a sharp algebraic model in terms of braid groups. In our case (graphs), the key information comes from a detailed control of the topological combinatorics associated to graph configuration spaces (both in the ordered and unordered contexts) provided by Farley-Sabalka's discrete gradient field on Abrams' homotopy model. The generalized braid-group characterization is established in Section \ref{caracterizacion}. Section~\ref{proofdetails} is devoted to proving Theorems~\ref{thirdmainintro} and \ref{secondmainH},  while Section \ref{swedgei} addresses the proof of Theorem \ref{maintheoremH} after a revision of the needed facts on discrete Morse theory.

\medskip\noindent {\bf Acknowledgements:} The first author was partially supported by the FAPESP ``Projeto Tem\'atico-FAPESP Topologia Alg\'ebrica, Geom\'etrica e Diferencial''  grant no. 2022/16455-6 (S\~ao Paulo-Brazil). This work started during the  visit of the second author to  IME-USP, S\~ao Paulo, from April 22 to May 06, 2023, and continued during  the visit of first author to the  Mathematics Department of Cinvestav from June 3 to June 14, 2024. The visit of the second author was partially supported by  the FAPESP ``Projeto Tem\'atico-FAPESP Topologia Alg\'ebrica, Geom\'etrica e Diferencial''  grant no. 2016/24707-4 (S\~ao Paulo-Brazil). Both authors are greatly thankful to the hosting institutions for the invitations and the great hospitality during the visits. 

\section {Generalized braid-group characterization}\label{caracterizacion}
Let $H$  be a finite group with $n$ elements and  $\tau: H\times X\to X$  be a free left action of $H$ on a given topological space $X$, where we set
\begin{equation}\label{notaciondeaccion}
h\cdot x:=\tau(h,x).
\end{equation}
In this section we consider the natural  action
\begin{equation}\label{laaccion} 
 \star\colon S_H\times F_{H}(Y)\to F_{H}(Y),
 \end{equation}
where $S_H$ stands for the symmetric group of set theoretic bijections of $H$, and $F_{H}(Y)$ is an $H$-labeled form of the standard configuration space of $n$ ordered pairwise distinct points on an arbitrary space $Y$. In particular, using the Cayley-type monomorphism
\begin{equation}\label{cayleytypemorphism}
\iota:H\to S_{H},
\end{equation}
we get a restricted left action $H\times F_{H}(Y)\to F_{H}(Y)$.     We will see that the former action is a left action and   so will be the restricted  action. The main goal in this section is the establishment of an algebro-topological characterization for the failure of the Borsuk--Ulam property with respect to $H$  for homotopy classes of maps   $f:X \to   Y$.   The characterization is given in terms  of the fundamental 
 groups of various  spaces related to $X$,  $Y$ and  $F_H(Y)$,  at least  when  these spaces  are  $K(\pi, 1)$. We stress that our characterization works independently of the cardinality of $H$ or, even, if $H$ fails to be abelian. In the next section we apply these results to study the case where the spaces are certain graphs and  $H=\mathbb{Z}_n$, the cyclic group of order $n$.     

\medskip
 We start by setting some conventions. The group operation on $S_H$ is the standard composition of maps, in the understanding that the first factor acts first. Namely, given  $\alpha_1, \alpha_2 \in S_{H}$, we define
 \begin{equation}\label{firstsecond}
 \alpha_1 \alpha_2:=  \alpha_2\circ  \alpha_1.
 \end{equation}
Observe that, up to this point, the group structure of $H$ is not used, and $S_{H}$ is just the usual   symmetric group $S_n$ on  $n$ letters (recall that $n=|H|$) with the corresponding convention about the action of the factors. An isomorphism $S_n\simeq S_H$ can be constructed using an arbitrary  bijection between $\{1,2,...,n\}$ and $H$.
Likewise, the space $F_{H}(Y)$ consists  of all injective maps from  $H$  to  $Y$ with the compact-open topology. Again, as $H$ is discrete, a given bijection between $\{1,2,\ldots,n\}$ and $H$ determines a homeomorphism between $F_{H}(Y)$ and the standard configuration space of $n$-tuples on $Y$. In such a setting, and using the notation (\ref{notaciondeaccion}), the map (\ref{laaccion}) is given by
 $$\alpha\star \beta=\beta \circ \alpha,$$ the composite of two functions. Lastly, we consider the Cayley-type map (\ref{cayleytypemorphism}) whose value at an element $h\in H$ is the bijection  $\iota(h): H \to H$ given by  $\iota(h)(h_1)=h_1h$ for all $h_1\in H$. The verification of the following fact is straightforward:

\begin{lemma}\label{action} 
The map $\star: S_{H}\times F_{H}(Y) \to F_{H}(Y)$  is a free left action, while $\iota: H  \to S_{H}$ is a group monomorphism.
\end{lemma}

Suppose that  $f:X\to Y$ is a map such that $f$ restricted to any   orbit of the free $H$-action on $X$ is injective, so that the homotopy class determined by $f$ does not satisfy the Borsuk--Ulam property with respect to the given $H$-action. Under  this  hypothesis 
we  have  a  well   defined  map   $F_f:X \to F_{H}(Y)$  given  by the formula  $F_f(x)(h)=f(hx)$.  Although straightforward to verify, the following  is a key fact.

\begin{lemma}\label{equivemb} Under the conditions above, the map $F_f:X\to F_H(Y)$ is equivariant with respect to the monomorphism (\ref{cayleytypemorphism}),  the free left action of $H$ on $X$ and the free left action of  $S_{H}$ on $F_{H}(Y)$.
\end{lemma} 
\begin{proof}   From the definitions, for arbitrary elements $x\in X$ and $h,h_1\in H$ we have
\begin{align*}
F_f(hx)(h_1)&=f(h_1(hx))=f((h_1h)x)=
 F_f(x)(h_1h)\\
 &=   F_f(x)(\iota(h)(h_1))= (F_f(x) \circ \iota(h))(h_1)=(\iota(h)\star F_f(x))(h_1).
\end{align*}
 So  $F_f(hx)=\iota(h)\star F(x)$, as asserted.  
\end{proof}

Actually, the existence of a map $F\colon X\to F_H(Y)$ which happens to be equivariant in the sense of Lemma \ref{equivemb} captures the failure of the Borsuk--Ulam property for the triple $(X,\tau,Y)$. Explicitely, such an equivariant map $F$ determines a map $f\colon X\to Y$ given by $f(x)=F(x)(e)$, where $e\in H$ is the neutral element. We then have:

\begin{lemma}\label{reciproco} For arbitrary elements $h\in H$ and $x\in X$, we have $f(hx)=F(x)(h)$. In particular $f$ is injective when restricted to any $H$-orbit in $X$ and, in fact, $F=F_f$.
\end{lemma}
\begin{proof}
This is again a straightforward computation:
$$
f(hx)=F(hx)(e)=\left( \iota(h)\star F(x) \right)(e)=\left( F(x) \circ \iota(h) \right)(e)=F(x)(h),
$$
where the second equality holds in view of the equivariant property assumed for $F$.
\end{proof}

\begin{remark}\label{ida}{\em
Let $X/\tau$ be the quotient of $X$ by the action of $H$, 
$p_{\tau}: X\to X/\tau$ the quotient   map, 
 $D_{H}(Y)$ the quotient  of $F_H(Y)$ by the action of  the  subgroup $\iota(H)\subset S_H$ and   $p_H:F_H(Y)  \to D_{H}(Y)$    the  quotient map.  Given  maps $f$  and  $F$  as above (so that the homotopy class of $f$ does not have the Borsuk--Ulam property with respect to $\tau$), we have the (topological) commutative diagram
 	\begin{equation}\label{techniqueI}\xymatrix{
X \ar[r]^{F\hspace{4mm}}  \ar[d]_{p_{\tau}} & F_H(Y)  \ar[d]^{p_{H}}\\
X/\tau \ar[r]^{\overline F\hspace{3mm}}  & D_{H}(Y)} 
\end{equation}
where  $\overline F$ is  the quotient of the equivariant map $F$. Fix a base point $x_1\in X$ and take the corresponding images as base points for the rest of the spaces in (\ref{techniqueI}). Note that the base point $F(x_1)\in F_H(Y)$ has ``$e$-coordinate'' $F(x_1)(e)=f(x_1)$. Indeed, by Lemma \ref{reciproco}, we have in fact
\begin{equation}\label{recupera}
\text{ev}_e\circ F=f,
\end{equation}
where  $\text{ev}_e\colon F_H(Y)\to Y$ is the map evaluating at the neutral element $e\in H$. 
Then set $$P_H(Y){}:={}\pi_1(F_H(Y), F(x_1)) \quad\text{and}\quad  B_{H}(Y){}:={}\pi_1(D_{H}(Y), p_H\circ F(x_1)),$$   and observe that,
since $p_\tau$ and $p_H$ 
are covering maps,   we obtain short exact sequences of groups
$$1 \to \pi_1(X,x_1) \stackrel{p_{\tau\#}}\longrightarrow \pi_1(X/\tau,p_\tau(x_1)) \stackrel{\theta_\tau}\longrightarrow H\to 1 \quad\text{and}\quad1 \to P_H(Y) \stackrel{p_{H\#}}\longrightarrow B_{H}(Y) \stackrel{\theta_H}\longrightarrow \iota(H) \to 1, $$
where $\theta$-arrows stand for the corresponding classifying maps, together with the (algebraic) commutative diagram
	\begin{equation}\label{techniqueII}\xymatrix{
\pi_1(X,x_1) \ar[r]^{\varphi}  \ar@{^{(}->}[d]_{p_{\tau\#}} & P_H(Y)  \ar@{^{(}->}[d]^{p_{H\#}}\\
\pi_1(X/\tau,p_\tau(x_1)) \ar[r]^{\hspace{7mm}\psi} \ar@{>>}[d]_{\theta_{\tau}} & B_{H}(Y) \ar@{>>}[d]^{\theta_H}\\
H \ar[r]^{\iota}_{\cong} & \iota(H),
}\end{equation}
where ${\varphi} =F_{*}$ and  ${\psi} ={\bar F}_{*}$ are the maps induced at the fundamental-group level.
}\end{remark}
 
 Under relatively mild conditions, the algebro-topological situations in (\ref{recupera}) and (\ref{techniqueII}) can be used to characterize the failure of the Borsuk--Ulam property we started with in Remark \ref{ida}. For instance: 

	\begin{theorem}\cite[Theorem 2.4]{Acta} \label{th:borsuk_braid}
		Let $(X, x_1)$ be a pointed, pathwise-connected $CW$-complex, and suppose that there exists a proper free cellular left action 
		$\tau\colon\thinspace H \times  X \to X$, where $H$ is a finite group. Consider the Cayley monomorphism  $\iota:H \to S_H$ in (\ref{cayleytypemorphism}). Let $(Y, y_1)$ be a pointed surface, where $Y$ is either $\mathbb{R}^2$  or a compact surface without boundary different from $\mathbb{S}^2$ and $\mathbb{RP}^2$, and equip  $F_H(Y)$  and $D_H(Y)$ with respective base points $y_1'\in F_H(Y)$ and $p_H(y_1')\in D_H(Y)$ satisfying $y_1'(e)=y_1$.
		\begin{enumerate}
			\item Let $\alpha \in [X,x_1;Y,y_1]$ be a pointed homotopy class. Suppose that there exist homomorphisms 
			$\varphi\colon\thinspace  \pi_1 (X, x_1) \to P_H (Y)$ and $\psi\colon\thinspace \pi_1 (X/\tau, p_\tau(x_1)) \to B_{H}(Y)$ fitting in the commutative diagram 			
			\begin{equation}\label{eq:diag_borsuk_braid_prime1}\begin{gathered}\xymatrix{
						\pi_1 (X,x_1) \ar@{.>}[rr]^{\varphi} \ar[d]_{p_{\tau\#}} \ar@/^0.9cm/[rrrr]^{\alpha_\#} 
						&& P_H(Y) \ar[d]^-{p_{H\#}} \ar[rr]^{\text{ev}_{e\#}} && \pi_1(Y,y_1) \\
						\pi_1(X/\tau,p_\tau(x_1)) \ar@{.>}[rr]^{\psi} \ar[d]_{\theta_\tau} && B_H(Y) \ar[d]^{\theta_H} && \\
					H\ar[rr]^{\iota}_{\cong}	& & \iota(H).& &
			}\end{gathered}\end{equation}			 
Then $\alpha$  does not have the Borsuk--Ulam property with respect to the $H$-action $\tau$. Conversely, if $H$ is an Abelian group and $\alpha$ does not have the Borsuk--Ulam property with respect to $\tau$, then there exist homomorphisms 
$\varphi: \pi_1(X,x_1) \to P_H(Y)$ and   
$\psi: \pi_1(X/\tau,p_{\tau}(x_1)) \to B_H(Y)$ rendering a commutative diagram (\ref{eq:diag_borsuk_braid_prime1}).			
			\item Let  $\beta\in [X, Y]$ be a free homotopy class, and let $\alpha  \in [X, x_1; Y, y_1]$ be a pointed homotopy class in the preimage of $\beta$ under the natural map\footnote{See \cite[Chapter III, Section 1]{White} for  more details  
about this map.} $[X, x_1; Y, y_1] \to [X, Y]$ which sends a based homotopy class into the corresponding free homotopy class. If $\beta$ has the Borsuk--Ulam property with respect to 
$\tau$, then so does $\alpha$. Conversely, if $H$ is an Abelian group and $\alpha$ has the Borsuk--Ulam property with respect to $\tau$, then so does $\beta$.
		\end{enumerate}
	\end{theorem}

A standard homotopy theoretic argument shows that  Theorem \ref{th:borsuk_braid}  holds true after replacing the third sentence in the paragraph preceding item 1 above  by the assumption that both $(Y, y_1)$ and  $(F_H(Y), z_1)$ are  pointed $K(\pi, 1)$  spaces.
 Furthermore, Lemmas \ref{equivemb} and \ref{reciproco} show that the  Abelianity hypothesis on  $H$  can be waived for the converse part. We thus get:   
	
	\begin{theorem} \label{th:borsuk_braidI1}
		Let $(X, x_1)$ be a pointed, pathwise-connected $CW$-complex, and suppose that there exists a proper free cellular left action 
		$\tau\colon\thinspace H \times  X \to X$, where $H$ is a finite group with Cayley embedding (\ref{cayleytypemorphism}).  Assume that both $(Y, y_1)$ and $(F_H(Y), y_1')$ are  pointed $K(\pi, 1)$ spaces with $y_1'(e)=y_1$. 		
		\begin{enumerate}
			\item
		Let 	$\alpha \in [X,x_1;Y,y_1]$ be a pointed homotopy class. The Borsuk--Ulam property with respect to $\tau$ fails for $\alpha$ if and only if there exist homomorphisms 
			$$\varphi\colon\thinspace  \pi_1 (X, x_1) \to P_H (Y) \quad\text{and}\quad\psi\colon\thinspace \pi_1 (X/\tau, p_\tau(x_1)) \to B_{H}(Y)$$ rendering a commutative diagram	(\ref{eq:diag_borsuk_braid_prime1}).		
						%
						%
						%
						%


			\item Let $\beta\in [X, Y]$ be a free homotopy class, and let $\alpha  \in [X, x_1; Y, y_1]$ be a preimage under the natural map $[X, x_1; Y, y_1]\to[X,Y]$. The  free homotopy  class  $\beta$ has the Borsuk--Ulam property with respect to $\tau$ 
			if and only if  
			 $\alpha$ has the Borsuk--Ulam property with respect to $\tau$.
		\end{enumerate}
	\end{theorem}

Since graphs and their configuration spaces are aspherical, the following result is a specialization of Theorem  \ref{th:borsuk_braidI1}.

	
	\begin{corollary}\label{th:borsuk_braidIcoro}
		Let $(\Gamma, x_1)$,  $(G, y_1)$ be  pointed connected finite graphs and  $\tau\colon H \times  X \to X$ 
		a proper free cellular left action of a finite group $H$ with Cayley embedding (\ref{cayleytypemorphism}).
		\begin{enumerate}
			\item\label{it:borsuk_braid1} A pointed homotopy class $\alpha \in [\Gamma,x_1;G,y_1]$ does not have the Borsuk--Ulam property with respect to $\tau$ if and only if   there exist homomorphisms
			$$\varphi\colon\thinspace  \pi_1(\Gamma, x_1) \to P_{H} (G)\quad\text{and}\quad 
			\psi\colon\thinspace \pi_1 (\Gamma/\tau, p_\tau(x_1)) \to B_{H}(G)$$
		 that fit in the commutative diagram
	\begin{equation}\label{techniqueVIII}
	\xymatrix{
 & \pi_1(G,y_1) \\
\pi_1(\Gamma, x_1) \ar[r]^{\varphi} \ar[ur]^{\alpha_{\#}} \ar@{^{(}->}[d] & P_{H}(G) \ar[u]_{\text{ev}_{e\#}} \ar@{^{(}->}[d]\\
\pi_1(\Gamma/\tau, p_{\tau}(x_1)) \ar[r]^{\hspace{5.5mm}\psi} \ar@{>>}[d]_{\theta_{\tau}} & B_{H}(G) \ar@{>>}[d]^{\theta_{H}}\\
H \ar[r]^{\iota}_{\cong} & \iota(H).
}\end{equation}
			\item Let $\beta\in [\Gamma, G]$ be a free homotopy class, and let $\alpha  \in [\Gamma, x_1; G, y_1]$ be a pointed homotopy class 
			which is mapped to $\beta$.  
			The class $\beta$ has the Borsuk--Ulam property with respect to $\tau$ if and only if   $\alpha$ does. In particular, if  the based homotopy classes $\alpha_1, \alpha_2$ 
			are  mapped to the same free homotopy class, then $\alpha_1$ has the Borsuk--Ulam property with respect to $\tau$ if and only if $ \alpha_2$ does.
		\end{enumerate}
	\end{corollary}

\section{The case $H=\mathbb{Z}_n$\label{proofdetails}}
In this section we apply  the results above to study  the Borsuk--Ulam  property for the case where $X=\Gamma$ is a finite connected graph, $Y$ is either a tree, a cycle, or the connected graph with two vertices and two edges one of which is a loop, and $H=\mathbb{Z}_n$   with free action $\tau:\mathbb{Z}_{n}\times \Gamma \to \Gamma$.

\subsection{Classification of $\mathbb{Z}_n$ free actions on graphs}

Some of the considerations in this section can be put forward for general groups. Yet, in view of our main goals, we work exclusively with $H=\mathbb{Z}_n$.  As in the previous section, $\theta_{\mathbb{Z}_n}:\pi_1( \Gamma/\tau)  \to \mathbb{Z}_n$ stands for the morphism corresponding to the classifying map of the covering $p_\tau\colon \Gamma\to\Gamma/\tau$.
Recall that:

\begin{definition}\label{equivalencia} Two free actions $\tau_i:\mathbb{Z}_n\times \Gamma_{i} \to \Gamma_{i}$ ($i\in\{1,2\}$) are equivalent  if there is a homotopy equivalence $k:\Gamma_{1}/\tau_1\to \Gamma_{2}/\tau_2$ such that $\theta_{\mathbb{Z}_n,2}\circ k_{\#}=\theta_{\mathbb{Z}_n,1}$.
 \end{definition}
 
 Since the pullback of a principal $\mathbb{Z}_n$ bundle under a homotopy equivalence yields a $\mathbb{Z}_n$-equivariant homotopy equivalence between total spaces, Definition \ref{equivalencia} amounts to setting $\tau_1$ and $\tau_2$ to be equivalent whenever there is a $\mathbb{Z}_n$-equivariant homotopy equivalence $h\colon\Gamma_1\to\Gamma_2$. In particular, the Borsuk--Ulam property behaves well with respect to the corresponding classification of actions:
 
 \begin{corollary}
 Under the conditions above, let $\alpha_i$ ($i\in\{1,2\}$) be homotopy classes of maps from $\Gamma_i$ to $G$ corresponding under $h$. Then $\alpha_1$ satisfies the Borsuk--Ulam property with respect to $\tau_1$ if and only if $\alpha_2$ satisfies the Borsuk--Ulam property with respect to $\tau_2$.
 \end{corollary}
 
We next spell out the corresponding classification of free $\mathbb{Z}_n$ actions on a given graph.
 
 \begin{proposition}\label{spellout} If a graph $\Gamma$ admits a free action of $\mathbb{Z}_n$ then the Euler characteristic $\chi(\Gamma)$ is divisible  by $n$.
 Conversely, if $\chi(\Gamma)$ is divisible  by  $n$, then there is a graph $\Gamma'$ homotopy equivalent to $\Gamma$ and admitting a free $\mathbb{Z}_n$ action.  
 \end{proposition}   
 \begin{proof}
 The first part of this result follows  because
  $\Gamma \to \Gamma/\mathbb{Z}_n$ is an $n$-fold covering, therefore we have 
   $\chi(\Gamma)=n \cdot \chi(\Gamma/\mathbb{Z}_n)$.  For the converse, let $\Gamma_1$ be any finite graph which has the homotopy type of the bouquet of 
$1-\chi(\Gamma)/n$ circles.  Since $\pi_1(\Gamma_1)$ is a free group of rank  $\geq 1$, there is a  surjective homomorphism  $\theta:\pi_1(\Gamma_1) \to \mathbb{Z}_n$. Let $\Gamma'$ be the covering of $\Gamma_1$ which corresponds to the $kernel$ of $\theta$.   Then, $\Gamma'$ is a finite graph where $\mathbb{Z}_n$ acts freely, it has the same Euler characteristic as $\Gamma$ (so the same homotopy type), and the result follows.
 \end{proof}
 
 The following algebraic lemma is the main ingredient in the determination of all the free $\mathbb{Z}_n$ actions. 
 Let $F(x_1,\ldots,x_r)$ be the free group of rank $r$ on a set $\{x_1,\ldots,x_r\}$. We will also write $F(r)=F(x_1,\ldots,x_r)$ when the base is not relevant.
 
 
 \begin{lemma}\label{th:Reibase} Given a surjective homomorphism $\theta:F(r) \to \mathbb{Z}_n$, there is a base $\{y_1,...,y_r\}$ of $F(r)$ with $\theta(y_1)$  a generator of $\mathbb{Z}_n$ and so that $\theta(y_i)=1$, the neutral element, for $1<i\leq r$.
 \end{lemma}
 \begin{proof} We first show the result for the case where $n=p^ r$ for $p$ a  prime number. The subgroups of $\mathbb{Z}_{p^ r}$ are precisely 
  $$\{1\}\subset p^ {r-1}(\mathbb{Z}_{p^ r})\subset\cdots\subset p^i(\mathbb{Z}_{p^ r})\subset\cdots\subset p(\mathbb{Z}_{p^ r}) \subset \mathbb{Z}_{p^ r},$$
  \noindent as $i$ runs from $r$ to $0$. Since $\theta$ is  surjective, the image of  some $x_i$ is a generator of $\mathbb{Z}_{p^ r}$, for otherwise the image of  $\theta$ would be contained in  $p(\mathbb{Z}_{p^ r})$.   Without loss of generality we may assume that $\theta(x_1)$ generates $\mathbb{Z}_{p^ r}$. 
 Then, for $i>1$ set $\theta(x_i)=\theta(x_1)^{r_i}$ for some integer $r_i$. Let $y_1=x_1$ and $y_i =x_ix_1^ {-r_i}$  for  $2\leq  i\leq r$. This is a base of $F(x_1,\ldots, x_r)$ with the required properties.
  
Now let $n$ be arbitrary with prime decomposition     $n=p_1^{n_1}p_2^ {n_2}\cdots p_s^ {n_s}$. Pick an isomorphism 
$$\mathbb{Z}_n \cong \mathbb{Z}_{p_1^{n_1}} \oplus   \mathbb{Z}_{p_2^ {n_2}}\oplus\cdots\oplus    
\mathbb{Z}_{p_{s-1}^{n_{s-1}}}\oplus  \mathbb{Z}_{p_s^{n_s}}$$  and let 
$q_j:  \mathbb{Z}_{n} \to \mathbb{Z}_{p_j^ {n_j}}$ denote  
the  projection onto the $j$th coordinate. Our proof is by  induction on the number of summands $s$ of $\mathbb{Z}_n$, with the previous paragraph grounding the induction. So assuming the result valid for $s$, we prove it for $s+1$. 
Let $$\theta:F(x_1,\ldots, x_r) \to  \mathbb{Z}_{p_1^ {n_1}} \oplus   \mathbb{Z}_{p_2^ {n_2}}\oplus\cdots\oplus    \mathbb{Z}_{p_{s-1}^{n_{s-1}}}\oplus  \mathbb{Z}_{p_s^{n_s}} \oplus  \mathbb{Z}_{p_{s+1}^{n_{s+1}}}\cong\mathbb{Z}_n$$ 
be a surjective homomorphism. As in the first paragraph of this proof, the surjectivity of 
$q_1\circ\theta$ implies that
the $q_1\circ\theta$ image of some $x_i$ must generate $ \mathbb{Z}_{p_1^ {n_1}}$. For simplicity of notation and without loss of generality, let us assume that this holds for $i=1$. Consider the set
$$J:=\left\{j\in\{1,2,\ldots,s+1\} \colon q_j(\theta(x_1)) \  is \ a \  generator\   of  \   \mathbb{Z}_{p_j^{n_j}}\right\}.$$ Once again, for simplicity we can safely assume $J=\{1,2,\ldots,u\}$ with $1\leq u\leq s+1$, and we let $J'$ stand for the complement of $J$ in $\{1,2,\ldots,s+1\}$.
 So we have a decomposition 
 $\mathbb{Z}_n=\mathbb{Z}_{m_1}\oplus \mathbb{Z}_{m_2}$,    where $\mathbb{Z}_{m_1}$ is the sum of the terms $$\mathbb{Z}_{p_j^{n_j}}$$ as $j$ varies in $J$, and $\mathbb{Z}_{m_2}$ is the corresponding sum as $j$ varies in $J'$ (note that $\mathbb{Z}_{m_2}$ might be the trivial group). Let $q_J:\mathbb{Z}_n\to  \mathbb{Z}_{m_1}$ and $q_{J'}:\mathbb{Z}_n\to \mathbb{Z}_{m_2}$ denote the natural projections. 
   For each $2 \leq  k  \leq s+1$, let $r_{k}$ be an integer   
 satisfying  the equation  $q_J\circ \theta(x_k)=(q_J\circ \theta(x_1))^{r_{k}}$  for  $2\leq k\leq s+1$. Then the elements
  $y_1=x_1$ and $y_j=x_jx_1^ {-r_j}$ for $j=2,\ldots,s+1$ yield a base $\{y_1,\ldots,y_r\}$ for $F(x_1,\ldots,x_r)$ satisfying
\begin{equation}\label{incigu}  
\theta(F(y_2,\ldots,y_r))\subseteq \mathbb{Z}_{m_2}.
\end{equation}
If $\mathbb{Z}_{m_2}=0$, we are done; otherwise, the choosing of $J'$ and the argument in the first paragraph of this proof shows that (\ref{incigu}) is in fact an equality. Since the number of   prime summands of $\mathbb{Z}_{m_2}$ is at most $s$, we can apply the inductive hypothesis. Namely, there is a base $z_2,\ldots,z_{r}$ of $F(y_2,\ldots,y_r)$ such that $\theta(z_2)$ is a generator of $\mathbb{Z}_{m_2}$, while $\theta(z_l)=1$ for $l>2$. Set $z_1=y_1$ and choose an integer $r_0$ satisfying the equation $q_{J'}\circ\theta(z_1)=\theta(z_2)^ {r_0}$, so $q_{J'}\circ\theta(z_1z_2^{-r_0})=1$. Therefore the element $w_1=z_1z_2^{-r_0}z_2$ has the property that $q_J\circ\theta(w_1)$ and $q_{J'}\circ\theta(w_1)$ generate $\mathbb{Z}_{m_1}$ and $\mathbb{Z}_{m_2}$, respectively. So $\theta(w_1)$ generates
 $\mathbb{Z}_n$. Let $r_0'$  be  an integer defined by the equation  $q_{J'}\circ\theta(z_2)=\theta(w_1)^{r_0'}$, and define $w_2=z_2w_1^ {-r_0}$. Then $\theta(w_2)=1$ and $w_1,w_2,z_3,....,z_r$ show yield the required base.  
    \end{proof} 

We close the section with a description of a base for the \emph{kernel} of $\theta$ in terms of the  base $\{y_1,...,y_r\}$ of $F(r)$ given by Lemma \ref{th:Reibase}.

\begin{lemma}\label{generatorsviaranks} The set 
$\{ \hspace{.5mm}y_1^n, \hspace{.3mm}y_1^{i}y_jy_1^{-i}  \ | \ i=0,1,\ldots,n-1  \text{ and } j=2,\ldots,r \}$ is a base of 
\hspace{.3mm}\emph{ker}$(\theta)$.
\end{lemma}
\begin{proof} This follows in a straightforward manner from the Reidemeister-Schreier method.
\end{proof}

\subsection{Trees and cycles}\label{ciclosyarboles}

Since the connected graph $\Gamma$ admits a free $\mathbb{Z}_n$ action (which we denote by $\tau$), Proposition \ref{spellout} gives $\chi(\Gamma)=-nm\leq0$ with $m\geq0$. In these conditions the rank of  the free group $\pi_1(\Gamma)$ is forced to be $nm+1$. In particular, {making use of a very specific base of $\pi_1(\Gamma)$ (referred to as ``a suitable identification'' in Theorem \ref{thirdmain} below), we obtain an explicit identification of the (either based or unbased) homotopy classes $[\Gamma,S^1]$ with $\mathbb{Z}^{nm+1}$.} 
In such terms, the following result gives a complete answer of the Borsuk-Ulam property for triples $(\Gamma,\tau_{\mathbb{Z}_n},G)$ and homotopy classes from $\Gamma$ to $G$ when $G$ is either a tree or a cycle (so $H=\mathbb{Z}_n$, the cardinality $n$ cyclic group).  

\begin{theorem}\label{thirdmain}
\begin{itemize}
\item[a)] If $G$ is homeomorphic to the interval $I$, then the Borsuk-Ulam property holds.
\item[b)] If $G$ is a tree not homeomorphic to the interval  $I$, then the Borsuk-Ulam property never holds.
\item[c)] If $G$ is homeomorphic to a circle $S^1$, then the Borsuk--Ulam property holds for most of the homotopy classes in $[\Gamma,S^1]$. Explicitly, under a suitable identification of $[\Gamma,S^1]$ with $\mathbb{Z}^{nm+1}$, where  $m+1$ is the rank of the free group  $\pi_1(\Gamma/\tau)$, the homotopy classes of maps $\Gamma\to S^1$ for which the Borsuk--Ulam property fails are precisely the $(nm+1)$-tuples
$$(p,p_1,p_1,\ldots,p_1,p_2,p_2,\ldots,p_2,\ldots,p_m,p_m,\ldots,p_m)$$
with $p$ congruent to $1$ mod n (and $p_1,\ldots,p_m$ arbitrary).
\end{itemize}
\end{theorem}
\begin{proof} \emph{a)} The configuration space $F_{n}(I)$ is the disjoint union of $n!$ open simplexes $\Delta_n$. Furthermore, any non-trivial permutation $\sigma\in  S_n$ has $\sigma(\Delta_n)\cap \Delta_n=\varnothing$ (see \cite{We}). Since  $\Gamma$ is connected, no $\mathbb{Z}_n$-equivariant  map $F:\Gamma \to  F_n(I)$ can exist, and the result follows directly from Lemma \ref{equivemb}.

\smallskip\noindent\emph{b)} {Under the current conditions there is a single} homotopy class of maps $\Gamma \to  G$. {By Corollary \ref{th:borsuk_braidIcoro}, it suffices to show that the corresponding (trivial) morphism induced at level of fundamental groups fits into a commutative diagram (\ref{techniqueVIII}).} The bottom {commutative square} of {(\ref{techniqueVIII})} admits {a} solution $\psi$, since $\pi_1(\Gamma/\tau,p_\tau(x_1))$ is a free group. {The middle commutative square of (\ref{techniqueVIII}) then admits a solution $\varphi$ because} the configuration space $F_{n}(G)$ is connected {(see for instance \cite[Theorem 2.7]{MR2701024}), so the two vertical downward maps on the right hand-side of (\ref{techniqueVIII}) assemble} a short exact sequence       
$$1 \to P_n(G) \to \pi_1(F_n(G)/\mathbb{Z}_n) \to  \mathbb{Z}_n \to  1.$$
{The top commutative triangle of (\ref{techniqueVIII}) holds trivially.}

\smallskip\noindent\emph{c)}  {In order to show the result we will use some information about the configuration space 
$F_n(S^1)$ as well as the action of $\mathbb{Z}_n\subset S_n$ on  $F_n(S^1)$ (see \cite{We}  for details). 
Choose a connected component $F_n(S^1)_{0}$ of $F_n(S^1)$, so we have $\pi_1(F_n(S^1)_0)\cong \mathbb{Z}$.
The subgroup $\mathbb{Z}_n$ acts on $F_n(S^1)_{0}$ with  $\pi_1(F_n(S^1)_{0}/\mathbb{Z}_n)\cong \mathbb{Z}$, and the homomorphism $\pi_1(F_n(S^1)_{0}) \to \pi_1(F_n(S^1)_{0}/\mathbb{Z}_n)$ is multiplication by~$n$.
Next, choose a base $y_1,\ldots,y_{m+1}$ of $\pi_1(\Gamma/\mathbb{Z}_n)$ as the one given by Lemma \ref{th:Reibase} for $\theta=\theta_\tau$, the epimorphism that corresponds to the classifying map for the cover $\Gamma\to \Gamma/\tau$, and choose the corresponding base of $\pi_1(\Gamma)$ given by Lemma \ref{generatorsviaranks}. Namely, the base consisting of the elements $e_1=y_1^n$ and $e_{i+1,\hspace{.3mm}j}=y_1^{i}y_jy_1^{-i}$, for $i=0,1,\ldots,n-1$ and $j=2,\ldots,m+1$. 
Consider now the commutative diagram
	\begin{equation}\label{techniqueV}\xymatrix{
 && \mathbb{Z} \\
\pi_1(\Gamma) \ar[rr]^{\varphi} \ar[urr]^{\alpha_{\#}} \ar@{^{(}->}[d] && \mathbb{Z} \ar[u]_{(p_1)_{\#}} \ar@{^{(}->}[d]^{\times n}\\
\pi_1(\Gamma/\tau) \ar@{.>}[rr]^{\psi} \ar[rd]_{\theta_\tau} && \mathbb{Z}  \ar[ld]^{\theta_2'} \\
& \mathbb{Z}_n&
}\end{equation}
where $\theta'_2$ corresponds to the classifying map for the covering $F_n(S^1)_0\to F_n(S^1)_0/\mathbb{Z}_n$. In order to prove the result we look for all possible morphisms $\psi:  \pi_1(\Gamma/\tau) \to  \mathbb{Z} $ rendering a commutative lower triangle in (\ref{techniqueV}). We can choose, for each $y_j$ with $m+1\geq j>1$, an arbitrary integer of  the form  $nk$ for arbitrary $k\in \mathbb{Z}$, so an arbitrary element  of  
$\pi_1(F_n(S^1)_0)\cong\mathbb{Z}$.  But   for a fixed $i$, the elements $e_{1,i},\ldots,e_{n,i}$ are mapped  to the same element, because the way how the $e_{n,i}$ were choosen.
 On the other hand $y_1$ must be mapped to an arbitrary integer $d$ with $d\equiv 1$ $mod(n)$.
 So the Borsuk-Ulam property does not hold if and only if the elements of the base $\{e_1, e_{1,1},\ldots,e_{n,1},\ldots, e_{1,m},\ldots,e_{n,m}\}$ are sent to the elements of $\mathbb{Z}$ of the form as stated in the lemma, and the result follows.}
\end{proof}

\section{The graph $S^1\vee I$}\label{swedgei}
We now address the Borsuk--Ulam property when the target graph has underlying topological space $S^1\vee I$. Here $I=[0,1]$, the unit interval with base point $0$. The two fundamental groups $P_{\mathbb{Z}_m}(S^1\vee I)$ and $B_{\mathbb{Z}_m}(S^1\vee I)$ needed in (\ref{techniqueVIII}) will be assessed through the use of discrete Morse Theory. Indeed, we use Farley-Sabalka's discrete gradient fields on Abrams' discrete homotopical models for the configuration space $F_m(S^1\vee I)$ and its quotient by the free $\mathbb{Z}_m$ action discussed in Section \ref{caracterizacion}.

\smallskip
Assuming the reader is familiar with Forman's discrete Morse theory \cite{MR1358614}, the section starts by reviewing the use of such techniques for assessing the graph braid groups we need. The foundations of the method can be found in Sections 2 and 3 of \cite{MR2171804}.

\subsection{Farley-Sabalka gradient field on Abrams homotopy model}\label{sectiononDMT}
For a graph $G$ with cells $e$ (vertices and edges), Abrams' discrete configuration space $DF_m(G)$ of $m$ non-colliding labeled cells is the cubical subcomplex of $G^m$ consisting of the cells $e_1\times\cdots\times e_m$ satisfying $\overline{e_i}\cap\overline{e_i}\neq\varnothing$ for $i\neq j$. We use the standard notation $(e_1,\ldots,e_m):=e_1\times\cdots\times e_m$. The obvious inclusion $DF_m(G)\hookrightarrow F_m(G)$ is a $\Sigma_m$-equivariant\footnote{The $\mathbb{Z}_m$ action, detailed in Section \ref{caracterizacion}, is spelled out in (\ref{leftaction}) below.} homotopy equivalence provided $G$ is $m$-sufficiently subdivided, i.e., when \emph{(i)} every path between two essential vertices $u\neq v$ touches at least $m$ vertices (including $u$ and $v$), and \emph{(ii)} every essential cycle (i.e.~one that is not homotopically trivial) at a vertex $u$ touches at least $m+1$ vertices (including $u$). Thus, under such conditions, we can ---and will--- use $DF_m(G)$ and $DF_m(G)/\mathbb{Z}_m$ as homotopy replacements for $F_m(G)$ and $F_m(G)/\mathbb{Z}_m$, respectively.

\smallskip
In what follows, $G$ stands for the $m$-sufficiently subdivided version of $S^1\vee I$ shown in (\ref{thegraph}) below, which has $2m$ vertices labeled from 0 to $2m-1$. The special edge between vertices $m-1$ and $2m-1$ is denoted by $a$. Let $T$ be the maximal tree of $G$ obtained by removal of the interior of $a$. The edge of $T$ joining vertices $i$ and $i-1$ ($i\in\{1,2,\ldots,2m-1\}$) will be denoted by $a_i$ and assigned the ordinal $\text{ord}(a_i)=i$, i.e., the ordinal of vertex $i$. By assigning the ordinal $\infty$ to $a$, we set a partial order on the cells of $G$, which restricts to linear orders both on vertices as well as on edges. Note that the $m$ ordinals of the coordinates of a cell of $DF_m(G)$ are pairwise different.
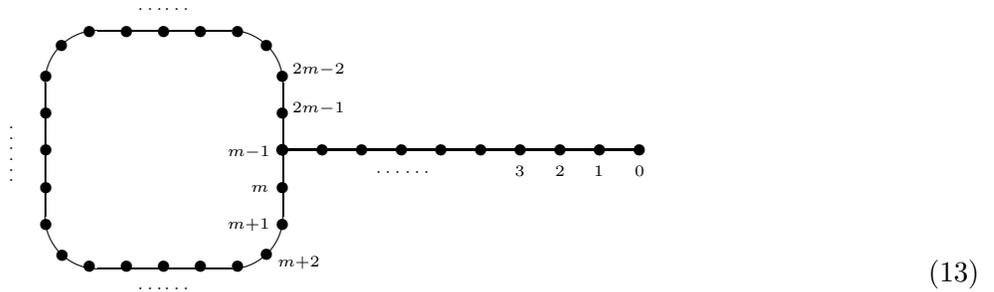
\begin{equation}\label{thegraph}
\begin{picture}(50,107)
\put(0,50){\line(1,0){135}}
\multiput(-3,47)(15,0){10}{$\bullet$}
\multiput(-3,47)(0,14){3}{$\bullet$}
\multiput(-3,47)(0,-14){3}{$\bullet$}
\multiput(-92.5,75)(0,-14){5}{$\bullet$}
\multiput(-76,92)(14,0){5}{$\bullet$}
\multiput(-76,3)(14,0){5}{$\bullet$}
\put(-86.4,7.3){$\bullet$}
\put(-86.6,86.6){$\bullet$}
\put(-9,86.6){$\bullet$}
\put(-9,7.5){$\bullet$}
\put(-45,50){\oval(90,90)[t]}
\put(-45,50){\oval(90,90)[b]}
\put(-21,47.5){\tiny$m{-}1$}
\put(133,40){\tiny$0$}
\put(117.5,40){\tiny$1$}
\put(102.9,40){\tiny$2$}
\put(87.7,40){\tiny$3$}
\put(35,40){\tiny$\cdots$}
\put(46,40){\tiny$\cdots$}
\put(-12,34){\tiny$m$}
\put(-21,20){\tiny$m{+}1$}
\put(-2,6){\tiny$m{+}2$}
\put(-55,-4){\tiny$\cdots$}
\put(-45,-4){\tiny$\cdots$}
\put(-55,102){\tiny$\cdots$}
\put(-45,102){\tiny$\cdots$}
\put(-104,50){\tiny$\vdots$}
\put(-104,38){\tiny$\vdots$}
\put(3.5,64){\tiny$2m{-}1$}
\put(3.5,79){\tiny$2m{-}2$}
\end{picture}
\end{equation}

\smallskip
Having fixed the graph $G$ and in order to simplify notation, $F_m$ and $F_m/\mathbb{Z}_m$ will stand, respectively, for Abrams' discrete homotopy model $DF_m(G)$ and its $\mathbb{Z}_m$ quotient $DF_m(G)/\mathbb{Z}_m$. Many of the considerations below apply both to $F_m$ and $F_m/\mathbb{Z}_m$ and, in such cases, we will simply write $A_m$ in order to refer to either $F_m$ or $F_m/\mathbb{Z}_m$.

\smallskip
Farley-Sabalka's gradient field $W_{FS}$ on $F_m$ is based on the concept of blocked vertices. Namely, a vertex coordinate $i=c_r$ of a cell $c=(c_1,\ldots,c_m)$ of $F_m$ is said to be blocked in $c$ if either $i=0$ or, else, $i-1$ intersects the closure of some $c_s$ with $s\neq r$. Otherwise $i$ is said to be unblocked in $c$. In the latter case, replacing the $r$-th coordinate $i$ of $c$ by the edge $a_i$ renders a cell $W_r(c)$ of $F_m$ having $c$ as a facet. With this preparation, a cell $c=(c_1,\ldots,c_m)$ of $F_m$ is:
\begin{itemize}
\item Critical provided all its vertex coordinates are blocked, and the only possible edge coordinate is $a$.
\item Redundant provided $c$ has an unblocked vertex coordinate $i=c_r$ so that no edge coordinate of $c$ has ordinal smaller than $i$. If $i$ is minimal with the property above, then the $W_{FS}$-pair of $c$ is set to be $W_r(c)$. Following standard notation: $c\nearrow W_r(c)$.
\item Collapsible provided $c$ has an edge coordinate $a_i=c_r$ so that no vertex coordinate of $c$ smaller than $i$ is unblocked in $c$. If $i$ is minimal with the property above, then the $W_{FS}$-pair of c is set to be $w_r(c)$, where $w_r(c)$ is the face of $c$ obtained by replacing the $r$-th coordinate $a_i$ of $c$ by the vertex $i$. Following standard notation: $w_r(c)\nearrow c$.
\end{itemize}

Farley-Sabalka's discrete gradient field on $F_m$ is $\Sigma_m$-equivariant, so it gives in the quotient a discrete gradient field on $F_m/\mathbb{Z}_m$. As reviewed below, these fields lead to a description of the corresponding fundamental groups. In each case, critical 1-cells (also referred to as critical edges) give generators, while critical 2-cells give relations. In the particular case of the graph (\ref{thegraph}), no critical 2-cells arise in $F_m$ or in $F_m/\mathbb{Z}_m$, so both fundamental groups are free. We next classify critical 1-cells (eventual generators) in each case. 

\medskip
For $F_m$, critical edges are given by $m$-tuples $c=(c_1,\ldots,c_m)$ with $c_i\neq c_j$ for $i\neq j$. Using the notation in (\ref{thegraph}) and the descriptions above, one of the $c_i$ equals $a$, while all other $c_i$'s are blocked vertices lying in $\{0,1,\ldots 2m-1\}$. For $b\in\{1,\ldots,m\}$, the critical edges of type $b$ are all possible tuples $c=(c_1,\ldots,c_m)$ with
\begin{equation}\label{tipo}
\{c_1,\ldots,c_m\}=\{0,1,\ldots,b-2,a,m,m+1,\ldots,2m-b-1\}.
\end{equation}
Note that vertices occupy the smallest possible slots on either side of $m-1$ ---they are blocked. Thus, a critical cell of type $b$ has $b-1$ vertices forming a pile blocked by $0$, while all other vertices form a pile blocked by $m-1$. For instance, all vertices $c_i$ are larger than $m-1$ forming a pile blocked by $m-1$ when $b=1$, while all vertices $c_i$ are smaller that $m-1$ forming a pile blocked by $0$ when $b=m$. 

\smallskip
By the $\Sigma_m$-equivariance of Farley-Sabalka's field, the exact same description and classification of critical edges holds for $F_m/\mathbb{Z}_m$, except that we then play with $\mathbb{Z}_m$-orbits of critical edges of $F_m$. As explained in Section \ref{caracterizacion}, the $\mathbb{Z}_m$-action turns out to be the restriction of the standard left action $\Sigma_m\times F_m\to F_m$ given by
\begin{equation}\label{leftaction}
\sigma\cdot(y_1,\ldots,y_m)=(y_{\sigma(1)},\cdots,y_{\sigma(m)}).
\end{equation}

\subsection{$\pi_1$ assessment}
This section computes the group inclusion and the classifying map on the right hand-side portion of (\ref{techniqueVIII}), i.e.,
\begin{equation}\label{maps}
\mathbb{Z}=\pi_1(G)\stackrel{p_1}{\longleftarrow}\pi_1(F_m)\stackrel{\iota}{\longhookrightarrow}\pi_1(F_m/\mathbb{Z}_m)\stackrel{\theta}{\longrightarrow}\mathbb{Z}_m.
\end{equation}

Farley-Sabalka's gradient field determine corresponding maximal forests $\mathcal{F}_m$ and $\mathcal{F}_m/\mathbb{Z}_m$, the former one in $F_m$ and the latter one in $F_m/\mathbb{Z}_m$. In each case, the forest consists of all vertices together with all redundant 1-cells of $A_m$. Under the canonical projection $F_m\to F_m/\mathbb{Z}_m$, $m$ of the trees of $\mathcal{F}_m$ project isomorphically to some tree of $\mathcal{F}_m/\mathbb{Z}_m$. As detailed in the paragraph containing (\ref{coset}) below, this results in a \emph{forest covering projection}
\begin{equation}\label{forestcovering}
\mathcal{F}_m\to\mathcal{F}_m/\mathbb{Z}_m.
\end{equation}

\begin{remark}\label{choosing}{\em
The \emph{Morsefied} Poincar\'e strategy for assessing $\pi_1(A_m)$ starts by extending the maximal forest in $A_m$ to a maximal tree. As shown by Farley and Sabalka, this can be achieved by selecting some critical edges that, when added to the maximal forest, render a maximal tree $\mathcal{T}_m$ for $A_m$. In turn, $\mathcal{T}_m$ is collapsed to some chosen base point $\star$ of $\mathcal{T}_m$, so that each unselected critical edge $x$ yields a generator, also denoted by $x$, of the (free in our case) group $\pi_1(A_m)$. Explicitly, before collapsing $\mathcal{T}_m$, a loop representing the generator $x$ is obtained by joining $\star$ to each of the end points of the edge $x$ through paths in $\mathcal{T}_m$. Although these joining paths can be chosen canonically, it becomes immaterial to explicitly describe them, as $\mathcal{T}_m$ will be collapsed. However, the selection of critical edges connecting the corresponding maximal forests into maximal trees should be done with care. Indeed, the process at the level of $\mathcal{F}_m$ must be compatible to the one for $\mathcal{F}_m/\mathbb{Z}_m$, thus resulting in a controlled description of the image of generators under the projection map inducing the group inclusion $\iota$ in (\ref{maps}). The following constructions pave the way for a suitable selection of the critical edges needed to connect the maximal forests $\mathcal{F}_m$ and $\mathcal{F}_m/\mathbb{Z}_m$ into corresponding maximal trees.}\end{remark}

\begin{definition}\label{orddef}
Think of a vertex $v=(v_1,\ldots,v_m)$ of $F_m$ as the bijection $\{1,\ldots,m\}\stackrel{v}{\to}\{v_1,\ldots,v_m\}$ given by $v(i)=v_i$. The permutation $\sigma_v$ associated to $v$ is defined to be the composition
$$
\{1,\ldots,m\}\stackrel{v}{\longrightarrow}\{v_1,\ldots,v_m\}\stackrel{\text{ord}}{\longrightarrow}\{1,\ldots,m\},
$$
where ord is the bijection that preserves order ---recall that, in the notation of (\ref{thegraph}), each $v_i$ is an integer. In other words, $\sigma_v$ is determined by the requirement that $v_{\sigma_v^{-1}(1)}<v_{\sigma_v^{-1}(2)}<\cdots<v_{\sigma_v^{-1}(m)}$.
\end{definition}

\begin{example}\label{act1}{\em
For $v=(8,2,5,9,4,3)$ we have
$$
\sigma_v=\begin{pmatrix}
\:\:1 & \mapsto & 8 &\mapsto & 5 \\
2 & \mapsto & 2 &\mapsto & 1 \\
3 & \mapsto & 5 &\mapsto & 4 \\
4 & \mapsto & 9 &\mapsto & 6 \\
5 & \mapsto & 4 &\mapsto & 3 \\
6 & \mapsto & 3 &\mapsto & 2 \:\:\\
\end{pmatrix}=(153462)$$
so that $v_2=2<v_6=3<v_5=4<v_3=5<v_1=8<v_4=9$.
}\end{example}

\begin{example}\label{act2}{\em
In connection to (\ref{leftaction}), if the vertex $v=(v_1,\ldots,v_m)$ of $F_m$ satisfies $v_1<v_2<\cdots<v_m$ then, for each permutation $\sigma\in\Sigma_m$, the permutation associated to $\sigma\cdot v$ is precisely $\sigma$.
}\end{example}

\begin{proposition}\label{permutationofcollapsibleedges}
Collapsible edges in $F_m$ connect vertices with the same associated permutation. In other words, any tree of the maximal forest $\mathcal{F}_m$ involves vertices with the same associated permutation.
\end{proposition}
\begin{proof}
Let $c$ be a collapsible edge in $F_m$. In the notation of (\ref{thegraph}), $c$ has coordinates $a_{v},u_1,\ldots,u_{m-1}$ (we thus ignore the actual order of coordinates in $c$), all of which belong to $T$ (the maximal tree of $G$). Recall that the coordinates of the redundant vertex $d$ of $F_m$ having $d\nearrow c$ are $v,u_1,\ldots,u_{m-1}$ where all $u_i$'s smaller than $v$ are in fact smaller than $v-1$ and form a pile of vertices blocked by $0$, while the rest of the $u_i$'s are larger than $v$ (and thus larger than $v-1$). The result follows as the vertex $d'$ of $c$ opposite to $d$ differs from $d$ in that the coordinate $v$ changes to $v-1$, so that the permutations associated to $d$ and $d'$ are identical.
\end{proof}

There are $m!$ critical vertices in $F_m$ (corresponding to all possible permutations of the blocked vertices $0,1,\ldots,m-1$). By \cite[Proposition 2.3]{MR2171804}, this means that there are $m!$ trees in the maximal forest $\mathcal{F}_m$. We thus get:
\begin{corollary}\label{treeslabelling}
All vertices with a fixed associated permutation assemble one of the $m!$ trees in $\mathcal{F}_m$.
\end{corollary}

We thus label trees of $\mathcal{F}_m$ by permutations, using $T_{\sigma}$ for the tree encoding the vertices with associated permutation $\sigma$. In particular, the $m$ trees $T_{(12\cdots m)^i\sigma}$, with $1\leq i\leq m$, get identified under (\ref{forestcovering}) to a single tree, denoted by $T_{[\sigma]}$, of the quotient forest $\mathcal{F}_m/\mathbb{Z}_m$. In other words, we label trees of the latter forest by $\Sigma_m/\mathbb{Z}_m$ cosets
\begin{equation}\label{coset}
[\sigma]=\{(12\cdots m)^i\sigma\colon1\leq i\leq m\}.
\end{equation}

\begin{remark}\label{rightaction}{\em
Since the cosets (\ref{coset}) are taken with a left indeterminacy, there is a well defined right action of $\Sigma_m$ into $\Sigma_m/\mathbb{Z}_m$, $\Sigma_m/\mathbb{Z}_m\times \Sigma_m\to\Sigma_m/\mathbb{Z}_m$, given by $[\sigma]\cdot\tau=[\sigma\cdot\tau]$.
}\end{remark}

Orient edges of $G$ from its vertex with the smaller index to its vertex with the larger index. For instance, $a$ is oriented from $m-1$ to $2m-1$. Edges of $A_m$ are then oriented accordingly. For instance, the type-$b$ critical edge
\begin{equation}\label{naturalmenteordenado}
\mathcal{O}_b:=(0,1,\cdots,b-2,a,m,m+1,\cdots,2m-b-1)
\end{equation}
of $F_m$ is oriented from its vertex $(0,1,\cdots,b-2,m-1,m,m+1,\cdots,2m-b-1)$ to its vertex $(0,1,\cdots,b-2,2m-1,m,m+1,\cdots,2m-b-1)$. We then say that these vertices are the source and the target, respectively, of $\mathcal{O}_b$.

\begin{lemma}\label{associatedpermutationsofacriticaedge}
Let a type-$b$ critical edge $x$ of $F_m$ have source vertex $\iota$ and target vertex $\varphi$. Then $\sigma_\varphi=\sigma_\iota c_b^{-1}$. Here $c_b$ stands for the cycle $(b\cdots m)\in\Sigma_m$ (for instance $c_1=(12\cdots m)$ and $c_m=1$, the identity permutation).
\end{lemma}
\begin{proof}
As in Definition \ref{orddef}, think of $x=(x_1,x_2,\ldots,x_m)$ as a bijection
$$x\colon\{1,2,\ldots,m\}\stackrel{x}\longrightarrow\{0,1,\ldots,b-2,a,m,m+1,\ldots,2m-b-1\}.$$
The permutations $\sigma_\iota$ and $\sigma_\varphi$ are obtained by, first, composing $x$ with corresponding bijections
\begin{align*}
\{0,1,\ldots,b-2,a,m,m+1,\ldots,2m-b-1\}
\longrightarrow\{0,1,\cdots,b-2,d(a),m,m+1,\cdots,2m-b-1\},
\end{align*}
where $d(a)=m-1$ in the case of $\sigma_\iota$, and $d(a)=2m-1$ in the case of $\sigma_\varphi$ and, then, composing the result with the corresponding ord bijection. The resulting permutations are
$$
j\mapsto x_j\mapsto\begin{cases} 
x_j \;\;(\text{for } 0\leq x_j\leq b-2) & \mapsto x_j+1\\
m-1 \;\;(\text{for } x_j=a) &  \mapsto b \\
x_j \;\;(\text{for } m\leq x_j\leq 2m-b-1) & \mapsto x_j-m+b+1
\end{cases}
$$ 
in the case of $\sigma_\iota$ while, for $\sigma_\varphi$, the resulting permutation is
$$
j\mapsto x_j\mapsto\begin{cases} 
x_j \;\;(\text{for } 0\leq x_j\leq b-2) & \mapsto x_j+1\\
2m-1 \;\;(\text{for } x_j=a) &  \mapsto m \\
x_j \;\;(\text{for } m\leq x_j\leq 2m-b-1) & \mapsto x_j-m+b.
\end{cases}
$$
A direct verification then gives that, as functions, $\sigma_\iota=(b,\ldots,m)\circ\sigma_\varphi$. The result then follows from the convention in (\ref{firstsecond}).
\end{proof}

\begin{remark}\label{associatedpermutationsofacriticaedgemodm}{\em
Lemma \ref{associatedpermutationsofacriticaedge} passes to the quotient $F_m/\mathbb{Z}_m$. Indeed, for a vertex $v=(v_1,\ldots,v_m)$ of $F_m$, an easy  direct verification gives $\sigma_{\mu\cdot v}=\mu\cdot\sigma_v$. In particular, the $\mathbb{Z}_m$-orbit $[v]$ has a well-defined associated permutation mod $\mathbb{Z}_m$ given by $[\sigma_v]$, which we denote by $\sigma_{[v]}$. Furthermore, for the type-$b$ critical edge $x$ in Lemma \ref{associatedpermutationsofacriticaedge} (starting at $\iota$ and ending at $\varphi$), we have that $c_1^i\cdot x$ is a type-$b$ critical edge starting at $c_1^i\cdot\iota$ and ending at $c_1^i\cdot\varphi$, and for which Lemma \ref{associatedpermutationsofacriticaedge} yields the second equality in
$$
c_1^i\cdot\sigma_\varphi= \sigma_{c_1^i\cdot\varphi}=\sigma_{c_1^i\cdot\iota}\cdot c_b^{-1}=c_1^i\cdot\sigma_\iota\cdot c_b^{-1}.
$$
This amounts to the fact that, for the corresponding critical edge $[x]$ of $F_m/\mathbb{Z}_m$ starting at $[\iota]$ and ending at $[\varphi]$, we have
\begin{equation}\label{pasadoalcociente}
\sigma_{[\varphi]}=\sigma_{[\iota]}\cdot c_b^{-1},
\end{equation}
where the product is the one in Remark \ref{rightaction}. In order to get the full quotient version of Lemma \ref{associatedpermutationsofacriticaedge}, it then suffices to declare the type of $[x]$ to be $b$, i.e., that of $x$.
}\end{remark}

\begin{remark}\label{recovering}{\em
The precise tuple expressions of the critical edge $x$ in Lemma \ref{associatedpermutationsofacriticaedge} and, therefore, of its end points $\iota$ and $\varphi$, can be recovered from the permutation $\sigma_\iota$ associated to the start of $x$, and from the type $b$ of $x$. Indeed, $b$ specifies, up to a permutation, the actual coordinates to be used, i.e. (\ref{tipo}), while $\sigma_\iota$ specifies the actual order of the coordinates in the tuple. Explicitly, $x=\sigma_\iota\cdot\mathcal{O}_b$ while $\iota=\sigma_\iota\cdot(0,1,\cdots,b-2,m-1,m,m+1,\cdots,2m-b-1)$ ---compare to Example \ref{act2}.
}\end{remark}

We are now ready to specify the selected critical edges in Remark \ref{choosing} in the case of $F_m$, i.e., the critical edges that will be added to the maximal forest $\mathcal{F}_m$ in order to connect it into a maximal tree. Namely, a type-$b$ critical edge $x=(x_1,\ldots,x_m)$ of $F_m$ is selected provided its first $b-1$ coordinates are minimal possible, i.e., $x_i=i-1$ for $1\leq i<b$, but the $b$-th entry is not minimal possible, i.e., $x_b\neq a$. Equivalently, $x=\sigma\cdot\mathcal{O}_b$ is selected precisely when $\sigma$ (the permutation associated the the start of $x$) satisfies $\sigma(i)=i$ for $1\leq i<b$ but $\sigma(b)\neq b$. In particular, the total number of type-$b$ critical edges that are selected is
\begin{equation}\label{numerodeselecccionadasdetipob}
(m-b)!(m-b).
\end{equation}
In particular, no critical edges of type $m$ are selected. The number in (\ref{numerodeselecccionadasdetipob}) plays a key role in the proof of:

\begin{proposition}\label{contractilextension}
Adding the selected critical edges of type $b$, for all $b\in\{1,\ldots,m-1\}$, to the maximal forest $\mathcal{F}_m$ renders a maximal tree in $F_m$.
\end{proposition}
\begin{proof}
Let $\Gamma$ stand for the graph that results by extending $\mathcal{F}_m$ through addition of the indicated critical edges, and then, for each $\sigma\in\Sigma_m$, collapsing to a point $z_\sigma$ the tree $T_\sigma$ of $\mathcal{F}_m$. Thus $\Gamma$ has $m!$ vertices and, in view of (\ref{numerodeselecccionadasdetipob}), $$m!-1=\sum_{b=1}^{m-1}(m-b)!(m-b)$$ edges. Since the Euler characteristic of $\Gamma$ is 1, once we prove that $\Gamma$ is connected, it will follow that $\Gamma$ is a tree and, hence, the result. We thus show that each $z_\sigma$ is connected in $\Gamma$ to $z_1$, the point corresponding to the identity permutation $1\in\Sigma_m$.

By Lemma \ref{associatedpermutationsofacriticaedge}, using selected critical edges of type $b=1$, we connect $z_{\sigma}$ to $z_{\sigma_1}$ for a permutation $\sigma_1\in\Sigma_m$ satisfying $\sigma_1(1)=1$. (For instance, $\sigma_1=\sigma$ if $\sigma(1)=1$.) Likewise, using selected critical edges of type $b=2$, we connect $z_{\sigma_1}$ to $z_{\sigma_2}$ for a permutation $\sigma_2\in\Sigma_m$ satisfying $\sigma_2(1)=1$ and $\sigma_2(2)=2$. In the last step of the process, we use selected critical edges of type $b=m-1$ in order to connect $z_{\sigma_{m-2}}$ to $z_{\sigma_{m-1}}$ for a permutation $\sigma_{m-1}\in\Sigma_m$ satisfying $\sigma_{m-1}(i)=i$ for $i\leq m-1$. But then $\sigma_{m-1}$ is forced to be the identity permutation.
\end{proof}

\begin{remark}\label{loscanonicosydemas}{\em
Before engaging into the next step ---selecting critical edges that connect $\mathcal{F}_m/\mathbb{Z}_m$ into a maximal tree for $F_m/\mathbb{Z}_m$---, it is convenient to choose canonical representatives of $\mathbb{Z}_m$-orbits of critical edges. First, the unique representative $\sigma$ in each coset (\ref{coset}) having $\sigma(1)=1$ will be called the \emph{canonical representative} of the coset. Canonical representatives allow us to identify $\Sigma_m/\mathbb{Z}_m$ with the set $\Sigma_{m-1}\subset\Sigma_m$ consisting of the permutations $\sigma\in\Sigma_m$ with $\sigma(1)=1$. Such a convention allows us to choose canonical representatives of vertices $[v]$ and of edges $[x]$ of $F_m/\mathbb{Z}_m$. Namely, $v$ is the canonical representative of $[v]$ provided $\sigma_v$ is the canonical representative of $[\sigma_v]$. In turn, $x$ is the canonical representative of a type-$b$ critical edge $[x]$ of $F_m/\mathbb{Z}_m$ provided the source of $x$ is the canonical representative of the source of $[x]$. In the latter situation, (\ref{pasadoalcociente}) shows that, when $b\geq2$, the target of $x$ is also the canonical representative of the target of $[x]$. Note that canonical representatives are given by the tuples whose first coordinate is smallest possible.
}\end{remark}

The selection of the connecting critical edges in Remark \ref{choosing} in the case of $F_m/\mathbb{Z}_m$ is formalized in terms of canonical representatives. Namely, while no type-$1$ critical edge is selected\footnote{This situation plays a subtle role in the description of the group monomorphism in (\ref{maps}). See the discussing following Proposition \ref{groupmonoinmostofthecases}.}, a type-$b$ critical edge with $b\geq2$ is selected provided its canonical representative $(x_1,\ldots,x_m)=\sigma\cdot\mathcal{O}_b$ satisfies $x_i=i-1$ for $1\leq i<b$ and $x_b\neq a$, i.e., provided the canonical representative is selected in the case of $F_m$. Thus, while the condition $\sigma(1)=1$ is automatic since we have chosen the canonical representative, the selection condition amounts to the conditions $\sigma(i)=i$ for $1\leq i<b$ and $\sigma(b)\neq b$.

\begin{remark}\label{condiciondecompatibilidad}{\em
By construction, the $\mathbb{Z}_m$-orbit of each selected critical edge of type at least 2 in $F_m$ is also selected in $F_m/\mathbb{Z}_m$. In fact, when $b\geq2$, we are selecting the same amount of type-$b$ critical edges in $F_m/\mathbb{Z}_m$ than in $F_m$. In particular, the total number of critical edges selected in the case of $F_m/\mathbb{Z}_m$ is
\begin{equation}\label{nuevototal}
(m-1)!-1=\sum_{b=2}^{m-1}(m-b)!(m-b).
\end{equation}
}\end{remark}

The compatibility condition in the first sentence of Remark \ref{condiciondecompatibilidad} addresses the requirements stressed in the second half of Remark \ref{choosing}. On the other hand, just as in the case of $F_m$, the second assertion of Remark \ref{condiciondecompatibilidad} leads to:

\begin{proposition}\label{contractilextensionmodm}
Adding the selected critical edges of type $b$, for all $b\in\{2,\ldots,m\}$, to the maximal forest $\mathcal{F}_m/\mathbb{Z}_m$ renders a maximal tree in $F_m/\mathbb{Z}_m$.
\end{proposition}
\begin{proof}
Consider the graph $\Gamma'$ analogue to the graph $\Gamma$ in the proof of Proposition \ref{contractilextension}. Looking at canonical representatives, we see that $\Gamma'$ has $(m-1)!$ edges and, by (\ref{nuevototal}), Euler characteristic $1$. The connectivity of $\Gamma'$ and, therefore, the conclusion then follow from the connecting process used in the second part of the proof of Proposition \ref{contractilextension}, except that this time the initial step of the process (with $b=1$) is not available nor needed as we work with canonical representatives.
\end{proof}

Theorem 2.5 in \cite{MR2171804} now yields:
\begin{corollary}\label{bases} 
\begin{enumerate}
\item A base for the free group $\pi_1(F_m)$ is given by the (based homotopy classes of the loops corresponding, under the considerations in Remark \ref{choosing}, to the) critical edges $(x_1,\ldots,x_m)=\sigma\cdot\mathcal{O}_b$ in $F_m$ of type $b$ $(1\leq b\leq m)$ satisfying the following two equivalent conditions:
\begin{align}
x_i=i-1 \mbox{ \ for \ } 1\leq i<b\:\:&\Longrightarrow\:\: x_b=a. \label{condicionparalabase} \\
\sigma(i)=i \mbox{ \ for \ } 1\leq i<b\:\:&\Longrightarrow\:\: \sigma(b)=b. \label{condicionparalabasebis}
\end{align}
(For instance, when $b=1$, the condition is simply $x_1=a$, i.e., $\sigma(1)=1$.)
\item A base for the free group $\pi_1(F_m/\mathbb{Z}_m)$ is given by the (based homotopy classes of the loops corresponding, under the considerations in Remark \ref{choosing}, to the) critical edges $[x]$ in $F_m/\mathbb{Z}_m$ of type $b$ $(1\leq b\leq m)$ whose canonical representative $(x_1,\ldots,x_m)=\sigma\cdot\mathcal{O}_b$ satisfies
\begin{itemize}
\item[(i)] $x_1=a$, i.e. $\sigma(1)=1$, when $b=1$.
\item[(ii)] $x_1=0$ and (\ref{condicionparalabase}), i.e., $\sigma(1)=1$ and (\ref{condicionparalabasebis}), when $b\geq2$.
\end{itemize}
(Beware: There is no restriction at all in (i), i.e., $x_1=a$ holds automatically for $b=1$, as we have taken the canonical representative. Likewise, the equality $x_1=0$ in (ii) is automatic.)
\end{enumerate} 
\end{corollary}

Propositions \ref{groupmonoinmostofthecases}--\ref{variacionparatipo4} below describe the images of base elements under the group inclusion in (\ref{maps}). The answer is as simple as possible for a few of the base elements of $\pi_1(F_m)$ of type $b\geq2$, namely, those that are canonical representatives of $\mathbb{Z}_m$-orbits:

\begin{proposition}\label{groupmonoinmostofthecases}
Any base element $[x]\in\pi_1(F_m/\mathbb{Z}_m)$ of type $b$ with $b\geq 2$ lies in the image of the group monomorphism $\pi_1(F_m)\hookrightarrow\pi_1(F_m/\mathbb{Z}_m)$ with preimage given by its canonical representative $(x_1,\ldots,x_m)\in\pi_1(F_m)$. In particular $\theta([x])=0$, where $\theta$ is the classifying map in (\ref{maps}).
\end{proposition}
\begin{proof}
Following the strategy in Remark \ref{choosing}, it suffices to observe that the canonical joining paths in the extended maximal tree for $F_m$ are made of collapsible edges and selected critical edges and, under the current hypotheses, map under the canonical projection mod $\mathbb{Z}_m$ into (perhaps non-canonical) joining paths in the extended maximal tree for $F_m/\mathbb{Z}_m$. This is so for collapsible edges in view of the $\Sigma_m$-equivariance of Farley-Sabalka's gradient field. The corresponding situation for selected edges is a consequence of the process in the second part of the proof of Proposition \ref{contractilextension}. Indeed, since $(x_1,\ldots,x_m)$ is a canonical representative, the first part of the process (with $b=1$) in that proof is not needed, so the assertion follows directly from the first part of Remark \ref{condiciondecompatibilidad}.
\end{proof}

Although selection of canonical representative settles a 1-1 correspondence between base elements of $\pi_1(F_m)$ of type 1 and base elements of $\pi_1(F_m/\mathbb{Z}_m)$ of type 1, the simple phenomenon in the proof of Proposition \ref{groupmonoinmostofthecases} for $b\geq2$ does no hold in the case $b=1$. The best way to state the new phenomenon is in terms of the notation in (\ref{naturalmenteordenado}). Namely, recall from Remark \ref{recovering} that, for a permutation $\sigma\in\Sigma_m$, $\sigma\cdot\mathcal{O}_b$ is the unique type-$b$ critical edge of $F_m$ starting at a vertex with associated permutation $\sigma$. Recall also that $\sigma\cdot\mathcal{O}_b$ has been selected to connect $\mathcal{F}_m$ if $\sigma(b)\neq b$ but $\sigma(i)=i$ for $1\leq i<b$; otherwise $\sigma\cdot\mathcal{O}_b$ becomes a base generator of $\pi_1(F_m)$. For instance, $\sigma\cdot\mathcal{O}_1$ is a base generator if and only if $\sigma(1)=1$, in which case $\sigma\cdot\mathcal{O}_1$ is the canonical representative of the corresponding base element $[\sigma\cdot\mathcal{O}_1]\in\pi_1(F_m/\mathbb{Z}_m)$. Bearing this in mind, the proofs of Propositions \ref{variacionparatipo1}--\ref{variacionparatipo4} can be given following the lines of the proof of Proposition \ref{groupmonoinmostofthecases}, except that now Lemma \ref{associatedpermutationsofacriticaedge} and its mod $\mathbb{Z}_m$ analogue in Remark \ref{associatedpermutationsofacriticaedgemodm} are used to keep track of portions of joining paths coming from selected critical edges of type 1. In particular, the permutations $c_b$ ($1\leq b\leq m$) in Lemma \ref{associatedpermutationsofacriticaedge} play an important role. Proof details are left as a straightforward exercise for the reader.  Note that most of the resulting $\mathbb{Z}_m$-representatives in (\ref{formulavariacionparatipo1})--(\ref{formulavariacionparatipo4}) below are not canonical.

\begin{proposition}\label{variacionparatipo1}
The image of a type-1 base element $\sigma\cdot\mathcal{O}_1\in\pi_1(F_m)$ (thus $\sigma(1)=1$) under the group monomorphism $\pi_1(F_m)\hookrightarrow\pi_1(F_m/\mathbb{Z}_m)$ is given by the product
\begin{equation}\label{formulavariacionparatipo1}
\prod_{i=1}^{m}\left[\rule{0mm}{4mm}\sigma c_1^{-i+1}\cdot \mathcal{O}_1\right].
\end{equation}
\end{proposition}

\begin{proposition}\label{variacionparatipo2}
The image of a base element $\sigma\cdot\mathcal{O}_b\in\pi_1(F_m)$ with $1<\sigma(1)<b$ under the group monomorphism $\pi_1(F_m)\hookrightarrow\pi_1(F_m/\mathbb{Z}_m)$ is given by the product
\begin{equation}\label{formulavariacionparatipo2}
\left(\prod_{i=1}^{\sigma(1)-1}\left[\rule{0mm}{4mm}\sigma c_1^{-i+1}\cdot \mathcal{O}_1\right]\right)^{\!\!-1}\hspace{-1.7mm}\left[\rule{0mm}{4mm}\sigma\cdot\mathcal{O}_b\right]\hspace{.3mm}\left(\prod_{i=1}^{\sigma(1)-1}\left[\rule{0mm}{4mm}\sigma c_b^{-1} c_1^{-i+1}\cdot \mathcal{O}_1\right]\right).
\end{equation}
\end{proposition}

\begin{proposition}\label{variacionparatipo3} The image of a base element $\sigma\cdot\mathcal{O}_b\in\pi_1(F_m)$ with $1<b=\sigma(1)$ under the group monomorphism $\pi_1(F_m)\hookrightarrow\pi_1(F_m/\mathbb{Z}_m)$ is given by the product
\begin{equation}\label{formulavariacionparatipo3}
\left(\prod_{i=1}^{\sigma(1)-1}\left[\rule{0mm}{4mm}\sigma c_1^{-i+1}\cdot \mathcal{O}_1\right]\right)^{\!\!-1}\hspace{-1.7mm}\left[\rule{0mm}{4mm}\sigma\cdot\mathcal{O}_{b}\right]\hspace{.3mm}\left(\prod_{i=1}^{m-1}\left[\rule{0mm}{4mm}\sigma c_{b}^{-1} c_1^{-i+1}\cdot \mathcal{O}_1\right]\right).
\end{equation}
\end{proposition}

\begin{proposition}\label{variacionparatipo4}
The image of a base element $\sigma\cdot\mathcal{O}_b\in\pi_1(F_m)$ with $1<b<\sigma(1)$ under the group monomorphism $\pi_1(F_m)\hookrightarrow\pi_1(F_m/\mathbb{Z}_m)$ is given by the product
\begin{equation}\label{formulavariacionparatipo4}
\left(\prod_{i=1}^{\sigma(1)-1}\left[\rule{0mm}{4mm}\sigma c_1^{-i+1}\cdot \mathcal{O}_1\right]\right)^{\!\!-1}\hspace{-1.7mm}\left[\rule{0mm}{4mm}\sigma\cdot\mathcal{O}_b\right]\hspace{.3mm}\left(\prod_{i=1}^{\sigma(1)-2}\left[\rule{0mm}{4mm}\sigma c_b^{-1} c_1^{-i+1}\cdot \mathcal{O}_1\right]\right).
\end{equation}
\end{proposition}

\begin{corollary}\label{auxi}
\begin{enumerate}
\item $\iota(c_1^r\cdot\mathcal{O}_m)=[\mathcal{O}_1]^{-r}\cdot[\mathcal{O}_m]\cdot[\mathcal{O}_1]^r$, for $0\leq r<m$.
\item $[\mathcal{O}_1]\cdot \iota(c_1^r\cdot \mathcal{O}_m)\cdot [\mathcal{O}_1]^{-1}=\iota(c_1^{r-1}\cdot \mathcal{O}_m)$, for $1\leq r<m$.
\end{enumerate}
\end{corollary}
\begin{proof}
Item 1 follows from Propositions \ref{groupmonoinmostofthecases}, \ref{variacionparatipo2} and \ref{variacionparatipo3}; item 2 follows from item 1.
\end{proof}
 
As in the last sentence of Proposition \ref{groupmonoinmostofthecases}, (\ref{formulavariacionparatipo1})--(\ref{formulavariacionparatipo4}) yield relations among $\theta$-images of base elements of $\pi_1(F_m/\mathbb{Z}_m)$. This leads to a full description (in Proposition \ref{theta} below) of the classifying map $\theta\colon\pi_1(F_m/\mathbb{Z}_m)\to\mathbb{Z}_m$ in (\ref{maps}). Easy cases come from Proposition \ref{groupmonoinmostofthecases}: 
\begin{equation}\label{tetaapartirde2}
\mbox{\emph{all base generators of $\pi_1(F_m/\mathbb{Z}_m)$ of type $b\geq2$ have a trivial $\theta$-image.}}
\end{equation}
So, we really care about the $\theta$-image of type-1 base elements $[\sigma\cdot\mathcal{O}_1]\in\pi_1(F_m/\mathbb{Z}_m)$. For simplicity, we use the notation $\theta[\sigma]:=\theta[\sigma\cdot\mathcal{O}_1]$, where square brackets are kept to emphasize that, in what follows, $\sigma$ might not be the canonical representative of $[\sigma]\in\Sigma_m/\mathbb{Z}_m$. Instead, we will write $\theta\langle\sigma\rangle$ rather than $\theta[\sigma]$ to stress that $\sigma(1)=1$, i.e., that $\sigma\in\Sigma_{m-1}$ is the canonical representative for $[\sigma]\in\Sigma_m/\mathbb{Z}_m$.

\medskip
Since elements in (\ref{formulavariacionparatipo1})--(\ref{formulavariacionparatipo4}) have trivial $\theta$-image, the observations in the previous paragraph yield
\begin{align}\label{simplificada1}
\sum_{i=1}^{m}\theta[\sigma c_1^{-i+1}]&\,\,=\,\,0,\\
\sum_{i=1}^{\sigma(1)-1}\theta[\sigma c_1^{-i+1}]&\,\,=\,\,
\begin{cases}
\displaystyle\rule{0mm}{3mm}\sum_{i=1}^{\sigma(1)-1}\theta[\sigma c_b^{-1} c_1^{-i+1}], & \text{when \ }1<\sigma(1)<b, \\
\displaystyle\rule{0mm}{9mm}\;\sum_{i=1}^{m-1}\theta[\sigma c_b^{-1} c_1^{-i+1}], & \text{when \ }1<\sigma(1)=b,\\
\displaystyle\rule{0mm}{10mm}\sum_{i=1}^{\sigma(1)-2}\theta[\sigma c_b^{-1} c_1^{-i+1}], & \text{when \ }1<b<\sigma(1).
\end{cases}\nonumber
\end{align}

Coming from Proposition \ref{variacionparatipo1}, (\ref{simplificada1}) is asserted only when $\sigma(1)=1$. But the latter restriction is unnecessary since we deal with the $\mathbb{Z}_m$-classes in (\ref{coset}). Nevertheless, in order to have full control of classes in $\Sigma_m/\mathbb{Z}_m$, it is convenient to rewrite the relations above in terms canonical representatives. For instance, (\ref{simplificada1}) becomes
$$
\sum_{i=1}^{m}\theta\langle c_1^{\sigma^{-1}(i)-1}\sigma c_1^{-i+1}\rangle\;=\;0
$$
for $\sigma(1)=1$, and more usefully for our purposes
{\scriptsize\begin{equation}\label{canonica2}
\sum_{i=1}^{\sigma(1)-1}\theta\langle c_1^{\sigma^{-1}(i)-1}\sigma c_1^{-i+1}\rangle\;=\;
\begin{cases}
\displaystyle\sum_{i=1}^{\sigma(1)-1}\theta\langle c_1^{\sigma^{-1}(i)-1}\sigma c_b^{-1} c_1^{-i+1}\rangle, & \text{when }1<\sigma(1)<b, \\
\displaystyle\rule{0mm}{9mm}\sum_{i=1}^{b-1}\theta\langle c_1^{\sigma^{-1}(i)-1}\sigma c_b^{-1} c_1^{-i+1}\rangle+\sum_{i=b}^{m-1}\theta\langle c_1^{\sigma^{-1}(i+1)-1}\sigma c_b^{-1} c_1^{-i+1}\rangle, & \text{when }1<\sigma(1)=b,\\
\displaystyle\rule{0mm}{9mm}\sum_{i=1}^{b-1}\theta\langle c_1^{\sigma^{-1}(i)-1}\sigma c_b^{-1} c_1^{-i+1}\rangle+\sum_{i=b}^{\sigma(1)-2}\theta\langle c_1^{\sigma^{-1}(i+1)-1}\sigma c_b^{-1} c_1^{-i+1}\rangle, & \text{when }1<b<\sigma(1).
\end{cases}
\end{equation}}

Cases having $\sigma(1)=2<b<m$ in the first line of (\ref{canonica2}), together cases having $b=2$ and $\sigma(1)=3$ in the third line of (\ref{canonica2}), lead to a full description of $\theta$.

\begin{proposition}\label{theta}
A base element $[\sigma\cdot\mathcal{O}_1]\in\pi_1(F_1/\mathbb{Z}_m)$ with canonical representative $\sigma\cdot\mathcal{O}_1$ (thus $\sigma(1)=1$) has $\theta$ image given by
$$\theta\langle\sigma\rangle=\left(\sigma^{-1}(2)-1\right)\,\theta\left([\mathcal{O}_1]\rule{0mm}{4mm}\right).$$ In particular $[\mathcal{O}_1]\in\pi_1(F_m/\mathbb{Z}_m)$ is forced to map under $\theta$ to a generator of $\mathbb{Z}_m$.
\end{proposition}
\begin{proof}
For $\sigma(1)=2<b\leq m$, the first line of (\ref{canonica2}) gives
$\theta\langle c_1^{\sigma^{-1}(1)-1}\sigma\rangle=
\theta\langle c_1^{\sigma^{-1}(1)-1}\sigma c_b^{-1}\rangle$
which, in terms of the bijection
$$
\Sigma_{m-1}^{(2)}:=\{\tau\in\Sigma_m \ | \ \tau(1)=2\}\stackrel{\phi}{\longrightarrow}
\{\tau\in\Sigma_m \ | \ \tau(1)=1\}=:\Sigma_{m-1}^{(1)}$$
given by $\phi(\sigma)=c_1^{\sigma^{-1}(1)-1}\sigma$, simplifies to $\theta\langle\sigma\rangle=\theta\langle\sigma c_b^{-1}\rangle$ for $\sigma\in\Sigma_{m-1}^{(1)}$ and $2<b\leq m$. Thus, $\theta\langle\sigma\rangle=\theta\langle\tau\rangle$
for $\sigma,\tau\in\Sigma_{m-1}^{(1)}$ whenever $\sigma^{-1}\tau$ lies in the subgroup generated by $c_3,\ldots, c_m$. In other words, for $\sigma\in\Sigma_{m-1}^{(1)}$, the value of $\theta\langle\sigma\rangle$ depends only on $\sigma^{-1}(2)\in\{2,3,\ldots,m\}$. Therefore, by setting $\theta_{\sigma^{-1}(2)}:=\theta\langle\sigma\rangle$, we define without ambiguity elements $\theta_2,\theta_3,\ldots,\theta_m\in\mathbb{Z}_m$. In these terms, our goal transforms into establishing the relation
\begin{equation}\label{objetivo}
\theta_i=(i-1)\theta_2, \text{ \ for }i\in\{2,3,\ldots,m\}.
\end{equation}
Now take $b=2$ and $\sigma\in\Sigma_m$ with $\sigma(1)=3$ in the third line of (\ref{canonica2}) to get
\begin{equation}\label{lasegunda}
\theta\langle c_1^{\sigma^{-1}(1)-1}\sigma c_2^{-1}\rangle=
\theta\langle c_1^{\sigma^{-1}(1)-1}\sigma\rangle+\theta\langle c_1^{\sigma^{-1}(2)-1}\sigma c_1^{-1}\rangle.
\end{equation}
The formula for $\phi$ in the previous paragraph also sets a bijection
$$
\Sigma_{m-1}^{(3)}:=\{\tau\in\Sigma_m \ | \ \tau(1)=3\}\stackrel{\phi}{\longrightarrow}
\{\tau\in\Sigma_m \ | \ \tau(1)=1\}=\Sigma_{m-1}^{(1)},$$
which allows us to write (\ref{lasegunda}) as 
$\theta\langle\tau c_2^{-1}\rangle=
\theta\langle\tau \rangle+
\theta\langle c_1^{\sigma^{-1}(2)-\sigma^{-1}(1)}\tau c_1^{-1}\rangle,\text{ \  for  } \tau\in\Sigma_{m-1}^{(1)}$.
In such an expression, $c_1^{\sigma^{-1}(2)-\sigma^{-1}(1)}\tau c_1^{-1}$ lies in $\Sigma_{m-1}^{(1)}$, which forces $\sigma^{-1}(2)-\sigma^{-1}(1)=\tau^{-1}(2)-1$. We thus get for $\tau\in\Sigma_{m-1}^{(1)}$ the relation
$\theta\langle\tau c_2^{-1}\rangle=
\theta\langle\tau \rangle+
\theta\langle c_1^{\tau^{-1}(2)-1}\tau c_1^{-1}\rangle$
or, in terms of the ``$\theta_i$''-notation, $$\theta_{\tau^{-1}(3)}=\theta_{\tau^{-1}(2)}+\theta_{\overline{\tau^{-1}(3)-\tau^{-1}(2)+1}},$$ where $\overline{\tau^{-1}(3)-\tau^{-1}(2)+1}$ stands for the integer in $\{2,3,\ldots,m\}$ congruent with $\tau^{-1}(3)-\tau^{-1}(2)+1$ mod $m$. In particular, $\theta_i=\theta_{i-1}+\theta_2$ for $3\leq i\leq m$, which leads to (\ref{objetivo}).
\end{proof}

\subsection{Proof of Theorem \ref{maintheoremH}}
Describing the map $p_1\colon\pi_1(F_m)\to \pi_1(G)$ in (\ref{maps}) is an easy task. First of all, by collapsing the tree $T$ of $G$, we see that a generator of $\pi_1(G)=\mathbb{Z}$ is given by (the canonical loop associated to) the edge $a$. On the other hand, the construction of Farley-Sabalka's field makes it clear that the only coordinate of dimension 1 in a collapsible edge in $F_m$ lies completely in $T$. Furthermore, no critical edge selected in $F_m$ to connect the maximal forest $\mathcal{F}_m$ can have the edge $a$ as first coordinate (cf.~paragraph following Remark \ref{recovering}). It follows that any of the joining paths in Remark \ref{choosing} has image under the canonical projection $F_m\to G$ entirely inside $T$, so that the $p_1$-image of a generator $x=(x_1,\ldots,x_m)$ of $\pi_1(F_m)$ in Corollary \ref{bases} depends exclusively on $x_1$. We thus get:

\begin{proposition}\label{proyeccion1}
The morphism $p_1$ in (\ref{maps}) sends a generator $x=(x_1,\ldots,x_m)$ of $\pi_1(F_m)$ in Corollary \ref{bases} into (the class of) $x_1$. Explicitly, $p_1(x)$ is the generator $a$ provided $x_1=a$; otherwise $p_1(x)$ vanishes.
\end{proposition}

The proof of Theorem \ref{maintheoremH} is based on Corollary \ref{th:borsuk_braidIcoro}. Namely, for an arbitrary fixed integer $k$, we show that the ``multiplication-by-$k$'' morphism $\mathbb{Z}=\pi_1(\Gamma)\to\pi_1(G)=\mathbb{Z}$ fits into a commutative diagram
	\begin{equation}\label{final}
	\xymatrix{
 & \pi_1(G) \\
\pi_1(\Gamma) \ar[r]^{\varphi} \ar[ur]^{k} \ar@{^{(}->}[d] & \pi_1(F_m) \ar[u]_{p_1} \ar@{^{(}->}[d]^{\iota}\\
\pi_1(\Gamma/\mathbb{Z}_m) \ar[r]^{\hspace{-1.3mm}\psi} \ar@{>>}[d]_{\theta'} & \pi_1(F_m/\mathbb{Z}_m) \ar@{>>}[d]^{\theta}\\
\mathbb{Z}_m \ar[r]_{=} & \mathbb{Z}_m.
}\end{equation}
for suitably defined morphisms $\varphi$ and $\psi$. 

\smallskip
Since $0=\chi(\Gamma)=m\cdot \chi(\Gamma/\mathbb{Z}_m)$, the graph $\Gamma/\mathbb{Z}_m$ has the homotopy type of the circle, so $\pi_1(\Gamma/\mathbb{Z}_m)=\mathbb{Z}$, and the group inclusion on the left hand-side of (\ref{final}) is given by multiplication by $m$. By Proposition \ref{theta}, a generator $g\in\pi_1(\Gamma/\mathbb{Z}_m)$ maps under $\theta'$ into
$\theta([\mathcal{O}_1])\in\mathbb{Z}_m$. Let $g'\in\pi_1(\Gamma)$ be the generator whose image in $\pi_1(\Gamma/\mathbb{Z}_m)$ is $m\cdot g$. Lastly, consider the base elements
$$\mbox{$z:=[\mathcal{O}_1]\in\pi_1(F_m/\mathbb{Z}_m)$ \ and \ $w_i:=c_1^{-i}\cdot\mathcal{O}_m\in\pi_1(F_m)$ \hspace{.3mm}for $1\leq i\leq m$.}$$

\smallskip
Given an integer $\ell$, the rule $g\mapsto z\cdot\iota(w_1^\ell)$ defines a morphism $\psi\colon\pi_1(\Gamma/\mathbb{Z}_m)\to\pi_1(F_m/\mathbb{Z}_m)$ rendering a commutative bottom square in (\ref{final}). By exactness of the vertical columns in (\ref{final}), the restriction of $\psi$ determines a morphism $\varphi\colon\pi_1(\Gamma)\to \pi_1(F_m)$ that renders a commutative middle square in (\ref{final}). It thus suffices to show that a suitable choice of $\ell$ renders a commutative top triangle in (\ref{final}). With this in mind, note that, by construction, $\iota\circ\varphi(g')=(z \cdot\iota(w_1^\ell))^m$. Using Corollary \ref{auxi} and neglecting the use of $\iota$, we then get
\begin{align}
(z w_1^\ell)^m&=(z w_1^\ell z^{-1}) \cdot (z^2 w_1^\ell z^{-2}) \cdots (z^{m-2} w_1^\ell z^{2-m})\cdot (z^{m-1} w_1^\ell z^{-m+1}) \cdot z^m \cdot w_1^\ell \nonumber\\
&=(z w_1 z^{-1})^\ell \cdot (z^2 w_1 z^{-2})^\ell \cdots (z^{m-2} w_1 z^{2-m})^\ell\cdot (z^{m-1} w_1 z^{-m+1})^\ell \cdot z^m \cdot w_1^\ell \nonumber\\
&=(c_1^{m-2}\cdot\mathcal{O}_m)^\ell\cdot
(c_1^{m-3}\cdot\mathcal{O}_m)^\ell\cdots
(c_1\cdot\mathcal{O}_m)^\ell\cdot
(\mathcal{O}_m)^\ell \cdot z^m \cdot (c_1^{m-1}\cdot\mathcal{O}_m)^\ell.\label{ellbases}
\end{align}
Each $c_1^i\cdot\mathcal{O}_m$ ($0\leq i<m$) in (\ref{ellbases}) lies in $\pi_1(F_m)$ and, in view of Proposition \ref{proyeccion1}, maps trivially under $p_1$, except for one of them which maps to the generator $a$. Therefore the $p_1$ image of (\ref{ellbases}) is $\ell a + p_1(z^m)$. The commutativity of the top triangle in (\ref{final}) is then assured by taking $\ell=k-j$ where $p_1(z^m)=ja$.

\begin{remark}Let $(X,\tau_H)$ be a pair satisfying the hypotheses of Theorem \ref{th:borsuk_braidI1}, with $H=\mathbb{Z}_n$. Based on Theorem \ref{th:borsuk_braidI1}, it is straightforward to verify that the following more general statements hold:
\begin{itemize}
\item[a)] Theorem \ref{thirdmainintro} replacing $(\Gamma, \tau_{\mathbb{Z}_n})$ by $(X,\tau_{\mathbb{Z}_n})$;
\item[b)] Theorem \ref{secondmainH} replacing $(\Gamma, \tau_{\mathbb{Z}_n})$ by $(X,\tau_{\mathbb{Z}_n})$ with $\pi_1(X)$  a free group of rank $1+mn$;
\item[c)] Theorem \ref{maintheoremH} replacing $(\Gamma, \tau_{\mathbb{Z}_n})$ by $(X,\tau_{\mathbb{Z}_n})$ with  $\pi_1(X)$ isomorphic  to $\mathbb{Z}$.
\end{itemize}
\end{remark}

{\sc \ 

Departamento de Matem\'atica, IME

Universidade de S\~ao Paulo

Rua do Mat\~ao 1010 CEP: 05508-090

S\~ao Paulo-SP, Brazil

{\tt dlgoncal@ime.usp.br}

\bigskip

Departamento de Matem\'aticas

Centro de Investigaci\'on y de Estudios Avanzados del I.P.N.

Av.~Instituto Polit\'ecnico Nacional n\'umero~2508, San Pedro Zacatenco

M\'exico City 07000, M\'exico.}

{\tt jesus@math.cinvestav.mx}
\end{document}